\documentclass[3p,10pt]{elsarticle}

\usepackage{amsmath,amsfonts}
\usepackage{amsthm}
\usepackage{stmaryrd}

\usepackage{graphicx}
\usepackage{color}
\usepackage{float}
\usepackage{booktabs}
\usepackage{siunitx}
\usepackage{multicol}
\usepackage{multirow}

\usepackage{soul}
\usepackage{algorithmic}
\usepackage{algorithm}

\usepackage{amsmath}
\usepackage{hyperref}

\definecolor{darkgreen}{rgb}{0.0, 0.42, 0.24}

\theoremstyle{plain}
\newtheorem{theorem}{Theorem}

\newtheorem{lemma}[theorem]{Lemma}

\newtheorem{remark}{Remark}

\newcommand\OO{\Omega}

\newcommand\GN{\Gamma_{\mathrm{N}}}
\newcommand\GD{\Gamma_{\mathrm{D}}}
\newcommand\RR{\mathbb{R}}

\newcommand\NN{\mathbb N}
\newcommand\PP{\mathbb{P}}

\newcommand\ff{{\boldsymbol f}}

\newcommand\gD{{\boldsymbol g}_{\mathrm{D}}}
\newcommand\gN{{\boldsymbol g}_{\mathrm{N}}}

\newcommand\nn{{\boldsymbol n}}

\newcommand\xx{{\boldsymbol x}}

\newcommand\uu{{\boldsymbol u}}
\newcommand\zz{{\boldsymbol z}}
\newcommand\vv{{\boldsymbol v}}
\newcommand\VV{\boldsymbol V}

\newcommand\ww{{\boldsymbol w}}
\newcommand\II{\boldsymbol{I}}

\newcommand{\defi}{{\,:=\,}}

\newcommand{\ddiv}{\mathrm{div}}
\newcommand{\zero}{\boldsymbol{0}}
\newcommand{\eps}{\varepsilon}
\newcommand{\balpha}{\boldsymbol{\alpha}}
\newcommand{\bsigma}{\boldsymbol{\sigma}}

\newcommand{\btau}{\boldsymbol{\tau}}

\newcommand{\HH}{\boldsymbol{H}}

\newcommand{\medio}{\dfrac{1}{2}}
\newcommand{\EE}{\mathcal{E}}

\newcommand{\TT}{\mathcal{T}}
\newcommand{\etu}{\eta^{\uu}}
\newcommand{\etsgma}{\eta^{\bsigma}}

\newcommand{\T}{\mathcal{T}}

\newcommand{\nh}[1]{{\left\vert\kern-0.25ex\left\vert\kern-0.25ex\left\vert #1 \right\vert\kern-0.25ex\right\vert\kern-0.25ex\right\vert}}

\newcommand{\nhc}[1]{{\left\vert\kern-0.25ex\left\vert\kern-0.25ex\left\vert #1 
    \right\vert\kern-0.25ex\right\vert\kern-0.25ex\right\vert^2}}

\newcommand{\dmean}[1]{{\left\{\kern-0.6ex\left\{ #1   \right\}\kern-0.6ex\right\}}}

\numberwithin{equation}{section}

\newcommand{\tr}{\mathrm{tr}\,}
\newcommand{\QQ}{\boldsymbol{Q}}
\newcommand{\PTk}{\mathcal{P}_T^k}
\newcommand{\PEk}{\mathcal{P}_E^k}

\usepackage{xcolor}

\usepackage[numbers]{natbib}

\newcommand{\RT}[1]{\mathbf{RT}_{#1}}                % Raviart--Thomas space
\newcommand{\tRT}[1]{\mathbb{RT}_{#1}}               % tensor Raviart--Thomas space
\newcommand{\calO}{\mathcal{O}}
\newcommand{\Hdiv}[1]{H({\rm{div}};#1)}              % Hdiv (vectorial)
\newcommand{\xiu}{\xi^{\uu}}
\newcommand{\xisigma}{\xi^{\bsigma}}

\begin{document}
%\linenumbers

\begin{frontmatter}

  \title{Convergence analysis of an augmented dual-mixed finite element method with skew-symmetric Nitsche stabilization for the generalized Oseen equations with mixed boundary conditions\tnoteref{t1}}

  \author[add1]{Cristian C\'arcamo}
  \ead{carcamo@wias-berlin.de}
  \address[add1]{Weierstra{\ss}-Institut f\"ur Angewandte Analysis  und  Stochastik,  Leibniz-Institut  im  Forschungsverbund  Berlin  e.V. (WIAS),  Berlin,   Germany}

  \author[add2]{Abner H. Poza}

  \ead{apoza@ucsc.cl}
  %\date{\today}
  \address[add2]{Departamento de Matem\'atica y F\'isica Aplicadas, Universidad Cat\'olica de la Sant\'isima Concepci\'on, Casilla 297, Concepci\'on, Chile, and, Grupo de Investigaci\'on en An\'alisis Num\'erico y C\'alculo Cient\'ifico, $GIANuC^2$, Concepci\'on, Chile}

  \author[add3]{Ramiro Rebolledo\corref{mycorrespondingauthor}}
  \cortext[mycorrespondingauthor]{Corresponding author}%
  \ead{ramirorebolledo@udec.cl}
          
  \address[mymainaddress]{Departamento de Agroindustrias, Facultad de Ingenier\'ia Agr\'icola, Universidad de Concepci\'on, Casilla 237, Chill\'an, Chile}
          
%          \tnotetext[t1]{\textcolor{darkgreen}{Preliminary version date: \today}}

\begin{abstract}
In this paper, building upon the ideas recently introduced in \cite{CCHP26}, we propose and analyze a new skew-symmetric Nitsche-augmented mixed scheme for the generalized Oseen equations with non-solenoidal convective fields, subject to mixed boundary conditions. The scheme is based on a continuous pseudostress-velocity augmented dual-mixed finite element formulation, in which the Dirichlet datum is imposed naturally, whereas the prescribed normal trace of the pseudostress is incorporated into the ${\mathbf H}(\operatorname{div};\OO)$ stress space. We establish the well-posedness of the continuous formulation by means of a classical lifting argument and its corresponding homogeneous problem. The essential boundary condition is then weakly enforced in the discrete scheme through suitable stabilizing Nitsche terms. We establish the existence and uniqueness of the solution to the discrete scheme, allowing the use of general conforming approximation spaces for both the pseudostress and the velocity. Moreover, in the case of homogeneous Dirichlet boundary conditions and general Neumann data, 
we establish the stability of the Nitsche scheme. Under mild regularity assumptions on the exact solution, we prove optimal convergence rates in an energy norm. Finally, numerical experiments in two and three dimensions confirm the theoretical convergence rates and assess the performance of the proposed scheme for non-solenoidal convective fields and solutions exhibiting boundary layers.

\end{abstract}


\end{frontmatter}

%\textcolor{red}{Llegue hasta ac\'a: Martes 6 de Octubre 2026 (17:33 hrs).}
%%%%%%%%%%%%%%%%%%%%%%%%%%%%%%%%%%%%% 
\section{Introduction}
%%%%%%%%%%%%%%%%%%%%%%%%%%%%%%%%%%%%% 
%
The generalized Oseen equations arise from the linearization of both the steady and unsteady Navier--Stokes equations, which describe the motion of viscous incompressible fluids in a wide range of physical settings.  The Navier--Stokes equations play a fundamental role in numerical analysis and scientific computing, with broad applications in engineering, geophysics, and biomedical sciences, among other fields (see, e.g.,~\cite{CuRo11,FoQuVe09,Kle97}). The numerical solution of these equations remains challenging due to the nonlinear convective term, 
 particularly at high Reynolds numbers, where convection may dominate diffusion.
In this context, generalized Oseen-type problems naturally arise at each step of fixed-point or Newton iterations applied to the Navier--Stokes equations (see, e.g.,~\cite{John16,Te01}).

Classical mixed finite element methods have become widely used tools for the numerical approximation of fluid flows, as they allow for the simultaneous approximation of primal variables, such as velocity and pressure, together with additional variables of physical interest, often referred to as dual variables, such as the pseudostress tensor or the velocity gradient, depending on the formulation. Furthermore, other relevant quantities, such as vorticity, can be accurately recovered through suitable postprocessing techniques~\cite{BookBrezziFortin2013,CaiTongVassilevskiWang2010,Gatica2014}. The well-posedness of these formulations typically relies on an inf-sup condition, also known as the Ladyzhenskaya--Babu\v{s}ka--Brezzi (LBB) condition, which must be satisfied by the discrete approximation spaces and may impose compatibility restrictions between the spaces associated with the different unknowns. An alternative approach consists of incorporating residual-based stabilization terms derived from the governing equations of the model problem, such as the conservation of momentum and mass equations, with the aim of allowing more general choices of finite element approximation spaces. 
This type the formulation is knows in the literature as augmented formulations (for an incomplete list, see for example~\cite{AnayaCaraballoRuizBaierTorres2023,CamanoGaticaOyarzuaTierra2016,FigueroaGaticaMarquez2009,GaticaGaticaMarquez2011}).

The imposition of boundary conditions in finite element formulations can be handled in different ways. For instance, a natural boundary condition in a primal formulation becomes an essential condition in the corresponding dual-mixed formulation~\cite{BookBrezziFortin2013,Gatica2014}. Consequently, such a condition, as in the case of Neumann boundary conditions, must be incorporated into the corresponding continuous function space, which may complicate the construction and practical implementation of suitable finite element subspaces satisfying the prescribed boundary conditions. One way to avoid imposing this constraint directly on the approximation space is to enforce the essential boundary condition weakly by means of a Lagrange multiplier~\cite{Ba73-2,Gatica2007,GaticaGaticaMarquez2014,GaticaGaticaSequeira2015}. However, this approach introduces an additional unknown and requires an additional finite element approximation space, together with a discrete inf-sup condition that must be satisfied. Another approach consists of homogenizing the essential boundary conditions by means of a suitable lifting. Although this technique allows the problem to be reformulated in terms of homogeneous boundary conditions, its discrete implementation may require constructing a lifting compatible with the finite element approximation spaces~\cite{BaBeBu25,BCG20,Barrios-Cascon-Gonzalez-Oseen}. For problems involving interface conditions between different subdomains, mortar finite element methods~\cite{BeMaPa92} provide a suitable alternative, particularly when the meshes do not match across the interface. These methods are also of interest in elasticity, especially for contact problems involving two deformable bodies~\cite{GuRaVi23}.

Another approach for the weak imposition of essential boundary conditions is the classical penalty method~\cite{babuska-73b,BabuskaZlamal1973}, which incorporates a mesh-dependent penalty term into the variational formulation without introducing additional unknowns. However, this method is not consistent, and its convergence properties depend on the choice of the penalty parameter (see, e.g.,~\cite{chouly2026,Gustafsson2025}). An alternative is Nitsche's method, originally introduced in~\cite{Nit71}, which weakly enforces essential boundary conditions while preserving consistency (see, e.g.,~\cite{chouly2026,Stenberg1995}). Under suitable assumptions, Nitsche's method allows optimal convergence rates to be achieved without increasing the number of discrete unknowns~\cite{Be02,Nit71}. Both symmetric and nonsymmetric formulations are available, the latter including skew-symmetric variants. While the classical symmetric formulation requires a sufficiently large penalty parameter to ensure coercivity, the stability of a penalty-free nonsymmetric Nitsche formulation has also been established for Poisson and convection--diffusion problems~\cite{Burman2012}. Furthermore, like the classical penalty method, Nitsche's method avoids introducing additional boundary unknowns and the associated saddle-point structure arising in Lagrange multiplier formulations.

In recent years, there has been growing interest in the use of Nitsche's method for the weak imposition of nonstandard boundary conditions in different variational formulations. Among the numerous contributions involving primal formulations, we first mention~\cite{ChibaSaito2023}, where Nitsche's method is employed to impose nonhomogeneous Robin boundary conditions for the Poisson problem. Also for the Poisson problem, but considering Dirichlet boundary conditions, a symmetric Nitsche formulation combined with least-squares and finite cell stabilization is proposed in~\cite{LarssonEtAl2022}, within the framework of unfitted finite element methods. In the context of linear elasticity, a Nitsche-based finite element approximation of frictionless unilateral contact conditions of Signorini type is studied in~\cite{BretinEtAl2023}, with particular emphasis on the convergence analysis arising in shape optimization. 

On the other hand, concerning mixed formulations, a Nitsche-based mixed finite element method is proposed in~\cite{AnayaArayaCapunayChouly2027} for the Poisson problem with mixed boundary conditions, in which the Neumann boundary condition is weakly enforced through a symmetric Nitsche formulation. The discrete scheme employs Brezzi--Douglas--Marini (BDM) elements, and the approach is further extended to the Poisson problem with Signorini boundary conditions. In a different context, two strategies inspired by contact mechanics are proposed in~\cite{GattiEtAl2024} for the weak imposition of nonlinear seepage boundary conditions in Richards' equation. The first approach relies on a strongly consistent penalty term inspired by Nitsche's method, whereas the second employs a hybridized formulation in which the boundary pressure is introduced as an additional unknown.

For fluid flow problems, three variants of Nitsche's method are studied in~\cite{araya2024} for the Stokes equations with slip boundary conditions. These schemes combine the weak imposition of essential boundary conditions through Nitsche terms with a stabilized velocity-pressure formulation based on equal-order $\mathbb P_k^n\times\mathbb P_k$ Lagrange finite elements. For the Brinkman problem, a symmetric Nitsche scheme based on a velocity-pressure formulation with grad-div stabilization is proposed in~\cite{AnayaCapunayCaraballoVellojin2026} to weakly impose Dirichlet and tangential boundary conditions. In the context of porous media flow, two mixed finite element formulations for the Darcy problem with Dirichlet and Neumann boundary conditions are analyzed in~\cite{BurmanPuppi2022}, using Raviart--Thomas elements conforming in $H(\ddiv;\OO)$. The first formulation is consistent and belongs to the class of Nitsche-type methods, whereas the second follows a penalty approach. Further applications of Nitsche's method to fluid flow problems can be found, among others, in~\cite{Bar26,Bar24,blank2018}.

More recently, a skew-symmetric Nitsche scheme based on an augmented dual-mixed finite element formulation has been proposed in~\cite{CCHP26} for the Brinkman problem with mixed Dirichlet--Neumann boundary conditions. The augmented formulation incorporates additional terms derived from the governing equations, allowing greater flexibility in the choice of finite element approximation spaces without requiring the discrete inf-sup compatibility conditions of the underlying mixed formulation.

Motivated by the preceding developments, in this work we build upon the frameworks introduced in~\cite{Barrios-Cascon-Gonzalez-Oseen,BCG20} to propose and analyze a new skew-symmetric Nitsche scheme for the generalized Oseen equations with non-solenoidal convective fields, mixed Dirichlet--Neumann boundary conditions, and non-null source terms. The method is based on an augmented dual-mixed pseudostress--velocity formulation, whose well-posedness at the continuous level is established by means of a suitable lifting of the Neumann datum. At the discrete level, the essential boundary condition is weakly enforced through consistent skew-symmetric Nitsche terms, using Raviart--Thomas elements for the pseudostress tensor and Lagrange elements for the velocity. We establish the existence and uniqueness of the discrete solution and prove the stability of the proposed scheme for homogeneous Dirichlet boundary conditions and general Neumann data. Furthermore, in contrast to~\cite{CCHP26}, we derive optimal convergence rates in an energy norm under weaker regularity assumptions on the exact solution. 
The proposed scheme does not introduce additional unknowns into the augmented formulation and can be readily implemented using the finite element library \texttt{FEniCS}~\cite{fenics15}. Finally, numerical experiments in two and three dimensions confirm the theoretical convergence rates and illustrate the performance of the method for non-solenoidal convective fields and convection-dominated flow regimes.

The remainder of this paper is organized as follows. In Section~\ref{model}, we introduce the model problem and the corresponding functional setting. We then recall the pseudostress -velocity formulation and establish the well-posedness of the resulting augmented variational formulation under suitable assumptions on the stabilization parameters. In Section~\ref{Nitsche}, we introduce the finite element approximation spaces and formulate the discrete scheme, in which the essential boundary condition is weakly enforced through suitable Nitsche terms. We also establish the existence and uniqueness of the discrete solution. Section~\ref{error} is devoted to the error analysis. We first establish the error orthogonality property and derive suitable approximation properties of the interpolation operators, which are subsequently used to derive an a priori error estimate in the energy norm. Finally, in Section~\ref{num}, we present numerical experiments that confirm the theoretical convergence rates and illustrate the performance of the proposed method.

%%%%%%%%%%%%%%%%%%%%%%%%%%%%%%%%%%%%   
\section{Model problem and preliminary results} \label{model}
%%%%%%%%%%%%%%%%%%%%%%%%%%%%%%%%%%%% 

\subsection{Function spaces}
Throughout this paper, let $\Omega\subset\RR^n$, $n\in\{2,3\}$, be a bounded Lipschitz domain with boundary $\Gamma\defi\GD\cup\GN$, where $\GD$ and $\GN$ are nonempty and disjoint. We denote by $\nn$ the outward unit normal vector to $\Gamma$.
We adopt standard notation for Sobolev spaces, inner products, and norms (see, e.g.,~\cite{EG21-FE1}). In particular, for any open set $\calO\subset\RR^n$, we denote by $L^2(\calO)$ the space of square-integrable functions on $\calO$. Moreover, $\|\cdot\|_{m,\calO}$ and $|\cdot|_{m,\calO}$ denote the standard norm and seminorm, respectively, in $H^m(\calO)$, $H^m(\calO)^n$, and $H^m(\calO)^{n\times n}$, for $m\geq 0$, with the convention $H^0(\calO)=L^2(\calO)$.  On the other hand, given a scalar field $q$ and a vector field $\vv=(v_i)_{i=1,\dots,n}$, we set the laplacian, gradient and divergence operators, respectively by
\[
\Delta q\defi \sum_{i=1}^n\frac{\partial^2 q}{\partial x_i^2}
\,,\quad
\nabla q
\defi\begin{pmatrix}
\dfrac{\partial q}{\partial x_i}
\end{pmatrix}_{i=1,\dots,n}
\quad\text{ and }\qquad
\nabla\cdot\vv\defi \sum_{i=1}^n\dfrac{\partial v_i}{\partial x_i}\,.
\]
Furthermore, the Laplacian of a vector field is defined componentwise, whereas the divergence of a tensor field is taken row-wise. We also introduce the Hilbert spaces
$\Hdiv{\OO}\defi\{\vv\in L^2(\OO)^n:\nabla\cdot\vv\in L^2(\OO)\}$
and
$\HH\defi\{\btau\in L^2(\OO)^{n\times n}:\nabla\cdot\btau\in L^2(\OO)^n\}$,
as well as the closed subspace
$\VV\defi\{\btau\in\HH:\btau\nn=\zero\text{ on }\GN\}$,
where the boundary condition is understood in the sense of normal traces.
In all these cases, we consider the standard $H(\mathrm{div})$-norm $\|\cdot\|_{\ddiv;\OO}$. For vector fields, we introduce the following Hilbert spaces:
%On the other hand, for vector fields, we consider the following Hilbert spaces:
\begin{eqnarray*}
  \QQ
  &\defi
  &H^1(\OO)^n\,,
  \\
  \HH^1_{\GD}(\OO)
  &\defi
  &\{\vv\in H^1(\OO)^n:\vv=\boldsymbol{0} \textrm{ on }\GD\}\,,
  \\
  \HH^1_{\GN}(\OO)
  &\defi
  &\{\vv\in H^1(\OO)^n:\vv=\boldsymbol{0} \textrm{ on }\GN\}\,,
  \\
  \HH^{1/2}_{00}(\GD)
  &\defi
  &\left\{\vv|_{\GD}: \,\vv \in \HH^1_{\GN}(\OO)\right\}\,,
  \\
  \HH^{1/2}_{00}(\GN)
  &\defi
  &\left\{\vv|_{\GN}: \,\vv \in \HH^1_{\GD}(\OO)\right\}\,.
\end{eqnarray*}
We denote by $\HH_{00}^{-1/2}(E) \defi (\HH_{00}^{1/2}(E))'$ with the topological dual of the Lions--Magenes space $\HH_{00}^{1/2}(E)$, endowed with the duality pairing $\langle\cdot,\cdot\rangle_{E}$, where $E\in\{\GD,\GN\}$  (for details, see \cite{SayasBrownHassell2020}).

%The latter space is equipped with the norm
%\begin{equation}
%\label{divnormtensor}
%\|\btau\|_{\ddiv;\OO}^2
%\defi
%\|\btau\|_{0,\OO}^2
%+
%\|\nabla\cdot\btau\|_{0,\OO}^2
%\qquad \forall\,\btau\in\HH.
%\end{equation}
%
%\vspace{5cm}
%
%
%Finally, we use $\zero$ to denote a generic null vector, %(including the null functional and operator),
%and use $C$ % and $c$
%with or without subscripts %, tildes, or hats,
%to denote generic constants independent of the discretization parameters, whose value may vary from one occurrence to another.
%
%
%Throughout the paper, we will consider the following Hilbert spaces:
%\begin{eqnarray*}
%  \HH
%  &\defi
%  &\bHdiv{\OO}\,,
%  \\
%  \QQ
%  &\defi
%  &H^1(\OO)^n\,,
%  \\
%  \VV
%  &\defi
%  &\{ \btau \in \HH: \btau \nn = \zero \textrm{ on }\GN\}\,,
%  \\
%  \HH^1_{\GD}(\OO)
%  &\defi
%  &\{\vv\in H^1(\OO)^n:\vv=\boldsymbol{0} \textrm{ on }\GD\}\,,
%  \\
%  \HH^1_{\GN}(\OO)
%  &\defi
%  &\{\vv\in H^1(\OO)^n:\vv=\boldsymbol{0} \textrm{ on }\GN\}\,,
%  \\
%  \HH^{1/2}_{00}(\GD)
%  &\defi
%  &\left\{\vv|_{\GD}: \,\vv \in \HH^1_{\GN}(\OO)\right\}\,,
%  \\
%  \HH^{1/2}_{00}(\GN)
%  &\defi
%  &\left\{\vv|_{\GN}: \,\vv \in \HH^1_{\GD}(\OO)\right\}\,.
%\end{eqnarray*}
%Furthermore, we denote by $\HH_{00}^{-1/2}(\GN) \defi (\HH_{00}^{1/2}(\GN))'$ with the topological dual of the Lions--Magenes space $\HH_{00}^{1/2}(\GN)$, endowed with the duality pairing $\langle\cdot,\cdot\rangle_{E}$, where $E\in\{\GD,\GN\}$  (for details, see \cite{SayasBrownHassell2020}).
%

\subsection{Problem model}

In this work, we consider the generalized  Oseen problem: \emph{Find a velocity vector field $\uu $ and a pressure scalar field $p$ such as}
\begin{equation}
  \label{Porig}
  (P) \qquad
  \left\{
    \begin{array}{rcll}
      - \nabla \cdot ( \nu \, \nabla \uu - p \II)
      + \theta\, \uu
      + (\nabla \uu )\balpha
      &=&\ff &\text{in }\OO\,,
      \\
      \nabla\cdot\uu
      &=& 0 & \text{in }\OO\,,
      \\
      \uu
      &=& \gD &\text{on }\GD\,,
      \\
      (-\nu\,\nabla \uu + p \II)\nn
      &=& \gN &\text{on }\GN\,,
    \end{array}
  \right.
\end{equation}
where $\nu>0$ is the kinematic viscosity, $\theta>0$ is a reaction coefficient, $\balpha\in W^{1,\infty}(\OO)^n$ denotes a convective velocity field, not necessarily solenoidal, $\II$ is the identity matrix of $\RR^{n \times n}$, 
$\gD \in \HH_{00}^{1/2}(\GD)$, $\gN \in\HH_{00}^{-1/2}(\GN)$, 
and $\ff\in L^2(\OO)^n$ is a body force vector. Moreover, we assume the following conditions on the convective term $\balpha$:

\begin{itemize}
 \item[(H1)] $\| \nabla \cdot \balpha \|_{\infty,\OO} < \varepsilon$ in $\OO$,
 where $\varepsilon>0$ is  small;
  \item[(H2)] $\balpha \cdot \nn \ge 0$ on $\GN$.
\end{itemize}

\begin{remark}
  Typically, the reaction parameter $\theta$ arises from a time discretization and is given by $\frac{1}{\Delta t}$, where $\Delta t$ denotes the time-step size.
  In turn, $\balpha$ is commonly obtained from a linearization of the Navier--Stokes equations and is therefore regarded as a known velocity field. Since the underlying flow is incompressible, it is natural to assume that $\balpha$ is nearly divergence-free, which motivates assumption (H1). Furthermore, assumption (H2) is naturally satisfied by the boundary conditions arising in many applications.
  \end{remark}

Next, at the continuous level, we follow the approach in~\cite{BCG20,Barrios-Cascon-Gonzalez-Oseen} to derive a mixed formulation of problem $(P)$. To this end, we introduce the pseudostress tensor
$  \bsigma\defi \nu\,\nabla\uu - p\,\II$
as an auxiliary unknown. In addition, given a tensor $\btau$, we denote its deviator by
$  \btau^d\defi \btau - \frac{\tr(\btau)}{n}\,\II\,.$ 
With these definitions, problem $(P)$ can be equivalently reformulated as follows: \emph{Find a tensor field $\bsigma$ and a vector field $\uu$ such that}
\begin{equation}
  \label{eq:2.1}
  \left\{
    \begin{array}{rcrl}
      - \nabla \cdot \bsigma
      + \theta\, \uu\,
      + (\nabla \uu )\balpha
      &=&  \ff & \text{in }\OO\,,
      \\
      \dfrac{1}{\nu}\,\bsigma^d -\nabla \uu
      &=&\zero & \text{in } \OO\,,
      \\
      \uu
      &=& \gD & \text{on } \GD\,,
      \\
      \bsigma\nn
      &=& -\gN & \text{on } \GN\,.
    \end{array}
  \right.
\end{equation}
The following lemma, whose proof is straightforward, establishes the equivalence between formulations $(P)$ and \eqref{eq:2.1}.
\begin{lemma}
Let $\bsigma=\nu\,\nabla\uu-p\,\II.$ 
Then, $(\uu,p)$ is a solution of problem $(P)$ if and only if $(\bsigma,\uu)$ is a solution of problem \eqref{eq:2.1}.
Moreover, the pressure variable can be recovered from $\bsigma$ through $p=-\frac{1}{n}\,\tr(\bsigma)$.
\end{lemma}
\begin{remark}
The previous result shows that, from the solution of the strong problem \eqref{eq:2.1} formulated in terms of the pseudostress and velocity, the pressure $p$ of the original problem $(P)$ can be recovered. Moreover, by means of a straightforward postprocessing procedure, other quantities of interest can also be obtained. In particular, since
$\bsigma^d=\nu\nabla\uu$,
the velocity gradient and the vorticity tensor are given, respectively, by
$\nabla\uu=\nu^{-1}\bsigma^d$
and
$\boldsymbol{\omega}=\dfrac{1}{2\nu}\left(\bsigma^d-(\bsigma^d)^t\right)$.
For further details, we refer to \cite{CaiTongVassilevskiWang2010}.
\end{remark}
The following lemma will be instrumental in the analysis of this work.
%The following lemma will be instrumental in the subsequent analysis. In particular, it will be used to establish the well-posedness of the discrete problem~\eqref{NitscheScheme}; see Theorem~\ref{coercivity-AN}.
%
\begin{lemma}
  \label{trilinearprop}
  Let $\ww,\vv\in\QQ$ and $\balpha \in H^1(\Omega)^n$. Then the following identity holds:
  \begin{equation}
    \label{trilinearprop-igualdad}
    \int_{\OO}(\nabla\ww)\balpha\cdot\vv
    =
    - \int_{\OO}(\nabla\vv)\balpha\cdot\ww
    - \int_{\OO}(\nabla \cdot \balpha)\ww\cdot\vv\,
    +\int_{\partial \OO} (\balpha\cdot\nn) \ww \cdot \vv\,.
  \end{equation}
  Moreover, if $\balpha \in L^{\infty}(\Omega)^n$, 
  then the following estimate holds for all $\ww\in\QQ$:
  \begin{equation}
    \label{trilinearprop-desigualdad0}
    \|(\nabla\ww)\balpha\|_{0,\Omega}^2
    \le
    n\,\|\balpha\|_{\infty,\OO}^2|\ww|_{1,\OO}^2\,.
  \end{equation}  
%  and there exists a positive constant $\Csob$, independent of $\ww$, $\vv$, and $\balpha$, such that
%  \begin{equation}
%    \label{trilinearprop-desigualdad}
%    \int_{\OO}((\nabla \cdot \balpha)\ww)\cdot\vv
%    \le
%    \Csob\, \|\nabla \cdot \balpha\|_{0,\OO}\|\ww\|_{1,\OO}\|\vv\|_{1,\OO}\,.
%  \end{equation}
\end{lemma}
\begin{proof}
Identity \eqref{trilinearprop-igualdad} follows from a straightforward integration by parts argument, while estimate \eqref{trilinearprop-desigualdad0} is an immediate consequence of the Cauchy--Schwarz inequality.
\end{proof}
By testing with suitable functions, integrating by parts, and adding appropriate residual terms from the momentum equations in \eqref{eq:2.1}, we obtain the following 
augmented mixed formulation, which is derived using arguments similar to those in \cite{BCG20,Barrios-Cascon-Gonzalez-Oseen}:
%\textcolor{red}{Llegue hasta ac\'a: Viernes 13 Julio 2026 (12:20 hrs) Ahora a comer :-)!!}
%Testing with suitable functions, integrating by parts and adding appropriate residual terms from the momentum equations on \eqref{eq:2.1}
%we (for details, see \cite{Barrios-Cascon-Gonzalez-Oseen}), we obtain the following augmented formulation: 
%., we arrive at the following dual-mixed variational formulation of problem \eqref{eq:2.1}: 
%To derive the variational formulation of \eqref{eq:2.2}, we combine the augmented formulation proposed for the Oseen problem in \cite{Barrios-Cascon-Gonzalez-Oseen} with the Nitsche-based treatment of boundary conditions proposed for the Brinkman equations in \cite{CCP26}, 
%obtaining the following augmented formulation: 
\emph{Find $(\bsigma,\uu) \in \HH \times \QQ$, with  $\bsigma\nn = -\gN$ on   $\GN$, such that}
\begin{equation}
  \label{eq:3.1} 
  A((\bsigma,\uu),(\btau,\vv))= F(\btau,\vv)\,,
\end{equation}
for all $(\btau,\vv) \in \VV \times \QQ $, where $A:(\HH \times \QQ)\times(\HH  \times \QQ) \longrightarrow \RR$ is the bilinear form defined by
\begin{align}
  \nonumber
  A\left( (\bsigma,\uu),(\btau,\vv) \right) \defi
  & \dfrac{1}{\nu}\,\int_{\OO} \bsigma^d:\btau^d
    + \int_{\OO} \uu \cdot (\nabla \cdot \btau)
    - \int_{\OO} \vv \cdot (\nabla \cdot \bsigma)
    + \theta\, \int_{\OO} \uu\cdot \vv  
    + \int_{\OO} (\nabla \uu)\balpha\cdot \vv 
  \\
  \nonumber
  & + \kappa_1\,\int_{\OO}\left(\nabla \uu- \dfrac{\bsigma^d}{\nu}\right):\left(\nabla\vv
    + \dfrac{\btau^d}{\nu}\right)
  \\
  & + \dfrac{\kappa_2}{2\theta}\,\int_{\OO} (\nabla \cdot \bsigma-\theta\, \uu-(\nabla \uu)\balpha) \cdot (\nabla \cdot \btau
    + \theta\, \vv +(\nabla \vv)\balpha) +\kappa_3 \, \int_{\GD} \uu \cdot \vv\,,
    \label{formA}
\end{align}
for all $(\bsigma,\uu),(\btau,\vv) \in \HH \times \QQ$, and $F:\HH \times \QQ \longrightarrow \RR$ is the linear form defined by
\begin{align}
  F\left(\btau,\vv\right) \defi
  & \langle \btau \nn, \gD \rangle_{\GD}
    + \int_{\OO}  \ff \cdot \vv
    - \dfrac{\kappa_2}{2\theta} \, \int_{\OO} \ff \cdot (\nabla \cdot \btau+\theta \,\vv
    +(\nabla \vv)\balpha) 
    +\kappa_3 \, \int_{\GD} \gD\cdot \vv\,,
\end{align}
for all $(\btau,\vv) \in \HH \times \QQ $, with $\kappa_1,$ $\kappa_2$ and $\kappa_3$, positive parameters which will be specified later.

The following technical results will be used to establish the coercivity of the bilinear form $A$ on the subspace $\VV\times\QQ$ of $\HH\times\QQ$, as well as the existence and uniqueness of the solution to the proposed discrete scheme.

%\textcolor{violet}{?`Ser\'ia conveniente a\~nadir un remark indicando qu\'e pasa con \eqref{formA} cuando $\theta=0$?}
\begin{lemma}
 \label{pseudocoercitivity}
 Assume that (H1)-(H2) hold and that the following conditions on the parameters are satisfied:
\begin{equation}
 \label{condicionkappa}
 0<\kappa_1<\nu\,, \qquad 0<\kappa_2<\min\left\{ 1 - \dfrac{\eps}{2\theta}, \dfrac{\theta \kappa_1}{n\|\balpha\|^2_{\infty,\OO}}\right\}\,, \qquad 
 \kappa_3 > \dfrac{\|\balpha \cdot \nn\|_{\infty,\GD}}{2}\,.
\end{equation}
Then, for all $(\btau,\vv) \in \HH \times \QQ$ there is a positive constant $C$, such that
\begin{align}
  \label{quasi-coercivity}
  A((\btau, \vv) ,(\btau, \vv) )
  \ge
  \, C \, \Bigg\{
    \dfrac{1}{\nu}\, \|\btau^d\|_{0,\OO}^2
       +\dfrac{\kappa_2}{2\theta}\,\| \nabla \cdot \btau\|^2_{0,\OO}
    + \theta\,\|\vv\|_{0,\OO}^2
    + |\vv|_{1,\OO}^2
  +\|\vv\|_{0,\GD}^2 \Bigg\}\,.
\end{align}
Moreover, there exists another positive constant $C$, depending on the physical parameters and the domain $\OO$, such that
 \begin{equation}
   \label{coercivityVQ}
  A((\btau, \vv) ,(\btau, \vv) )
  \ge
  \, C \, \Bigg\{
\|\btau\|_{\ddiv,\OO}^2
    + \|\vv\|_{1,\OO}^2
  +\|\vv\|_{0,\GD}^2 \Bigg\}\,,
\end{equation}
 for all  $(\btau,\vv) \in \VV \times \QQ$.
\end{lemma}

%
% Proof 
%

\begin{proof}
%The estimate   \eqref{quasi-coercivity} follows from a direct calculation. To this end, 
Let $(\btau,\vv)\in\HH\times\QQ$. Then, using the definition of the bilinear form $A(\cdot,\cdot)$ (cf. \eqref{formA}), 
H\"older and 
Young's inequalities, (H1)-(H2) and Lemma \ref{trilinearprop}, we obtain
\begin{align}
  \nonumber
 & A((\btau, \vv) ,(\btau, \vv) )
    \\
  \nonumber
  =&\,
     \dfrac{1}{\nu}\,\|\btau^d\|_{0,\OO}^2
     + \theta\,\|\vv\|_{0,\OO}^2
     + \int_{\OO}(\nabla\vv)\balpha\cdot\vv
     + \kappa_1\,
     \left(
     |\vv|_{1,\OO}^2 - \dfrac{1}{\nu^2}\, \|\btau^d\|_{0,\OO}^2
     \right)+
  \\
  \nonumber
  &\,
     \dfrac{\kappa_2}{2\theta}\,
    \left(
    \| \nabla \cdot \btau\|^2_{0,\OO} - \|\theta\,\vv+(\nabla\vv)\balpha\|^2_{0,\OO}
    \right)
    + \kappa_3\,\| \vv\|_{0,\GD}^2
  \\
  \nonumber
  \ge
  &\,
    \dfrac{1}{\nu}\left(1-\dfrac{\kappa_1}{\nu}\right)\|\btau^d\|_{0,\OO}^2
        +\dfrac{\kappa_2}{2\theta}\,\| \nabla \cdot \btau\|^2_{0,\OO}
    + \theta(1-\kappa_2)\,\|\vv\|_{0,\OO}^2- \medio\,
    \int_{\OO}(\nabla\cdot \balpha) \vv \cdot\vv+\\
  \nonumber
  &\, \medio \,\int_{\GD} (\balpha\cdot \nn) \vv\cdot \vv +\left( \kappa_1-\dfrac{\kappa_2\,n\|\balpha\|_{\infty,\OO}^2}{\theta}\right)\, |\vv|_{1,\Omega}^2
  +\kappa_3\,\|\vv\|_{0,\GD}^2\\
    \ge  \nonumber
  &\,
    \dfrac{1}{\nu}\left(1-\dfrac{\kappa_1}{\nu}\right)\|\btau^d\|_{0,\OO}^2
        +\dfrac{\kappa_2}{2\theta}\,\| \nabla \cdot \btau\|^2_{0,\OO}
    + \theta\left(1-\kappa_2-\dfrac{\eps}{2\theta}\right)\,\|\vv\|_{0,\OO}^2+\\
    %- \medio\,
    %\int_{\OO}(\nabla\cdot \balpha) \vv \cdot\vv +\\
  \nonumber
  &\, \left( \kappa_1-\dfrac{\kappa_2\,n\|\balpha\|_{\infty,\OO}^2}{\theta}\right)\, |\vv|_{1,\Omega}^2
  +\left( \kappa_3 - \dfrac{\|\balpha\cdot \nn\|_{\infty,\GD}}{2}\right)\,\|\vv\|_{0,\GD}^2\,,
  %
  % Final
  %
%  \ge&\,  \, C \, \Bigg\{
%    \dfrac{1}{\nu}\, \|\btau^d\|_{0,\OO}^2
%       +\dfrac{\kappa_2}{2\theta}\,\| \nabla \cdot \btau\|^2_{0,\OO}
%    + \theta\,\|\vv\|_{0,\OO}^2
%    + |\vv|_{1,\OO}^2
%  +\|\vv\|_{0,\GD}^2 \Bigg\}.   \nonumber
\end{align}
and the estimate~\eqref{quasi-coercivity} follows from the conditions on the parameters given in~\eqref{condicionkappa}. Moreover, the second lower bound~\eqref{coercivityVQ} follows by restricting the previous estimate to $\VV\times\QQ$ and applying~\cite[Lemma~4]{GaticaGaticaMarquez2014}.
\end{proof}
%
%
%

%\newpage

%
%It is also easy to prove, using \cite[Problem~(23)]{BaBeBu25}, that for
%$0 < \kappa_1 < \nu$, $\kappa_2\in (0,1/\alpha)$ and $\kappa_3 > 0$,  the problem~\eqref{eq:3.1} admits a unique solution $(\bsigma, \uu) \in \HH \times \textcolor{red}{\QQ}$.
%\textcolor{red}{
%\begin{lemma}
%If $0 < \kappa_1 < \nu$, $\kappa_2\in (0,1/\alpha)$ and $\kappa_3 \ge 0$, then problem~\eqref{eq:3.1} admits a unique solution $(\bsigma, \uu) \in \HH \times \QQ$.
%\end{lemma}
%%
%%
%%
%\begin{proof}
%The conclusion follows from arguments similar to those in Lemma~\ref{lemma1}, together with the equivalence between problem~\eqref{eq:3.1} and \cite[Problem~(23)]{BaBeBu25},  and \cite[Theorem~4]{BaBeBu25}.
%\end{proof}
%}
The following result guarantees the well-posedness of the variational formulation \eqref{eq:3.1} by means of standard lifting and homogenization arguments. For completeness, we provide the details.
\begin{lemma}
Assume the assumption ~\ref{pseudocoercitivity}. Then the variational formulation \eqref{eq:3.1} is well-posed.
\end{lemma}
\begin{proof}
We proceed as usual (see, for instance, \cite{BaBeBu25,CCHP26}) by considering a lifting of the Neumann datum $\gN$. Let $\bsigma_g\defi -\nabla\zz$, where $\zz\in\HH^1_{\GD}(\OO)$ is the unique solution of the problem:
$-\Delta\zz=\zero$ in $\OO$, $\zz=\zero$ on $\GD$, and $(\nabla\zz)\nn=\gN$ on $\GN$. It is clear that $\bsigma_h\in \HH$, $\nabla \cdot
\bsigma_h=0$ in $\OO$ and $\bsigma_h\nn=-\gN$ on $\GN$.
If we now define $\widetilde{\bsigma}\defi \bsigma-\bsigma_h$, then, from the strong formulation \eqref{eq:2.1}, we obtain the following problem in which the Neumann boundary condition has been homogenized:
\begin{equation*}
  \left\{
    \begin{array}{rcrl}
      - \nabla \cdot \widetilde{\bsigma}
      + \theta\, \uu\,
      + (\nabla \uu )\balpha
      &=&  \ff & \text{in }\OO\,,
      \\
      \dfrac{1}{\nu}\,\widetilde{\bsigma}^d -\nabla \uu
      &=&-  \dfrac{1}{\nu}\,\bsigma_h^d& \text{in } \OO\,,
      \\
      \uu
      &=& \gD & \text{on } \GD\,,
      \\
      \widetilde{\bsigma}\nn
      &=& \zero & \text{on } \GN\,.
    \end{array}
  \right.
\end{equation*}
A weak formulation of this problem is given by: \emph{Find $(\widetilde{\bsigma},\uu) \in \VV \times \QQ$, such that}
\begin{equation}
  \label{eq:2.7b} 
  A((\widetilde{\bsigma},\uu),(\btau,\vv))= F(\btau,\vv)\,,
\end{equation}
for all $(\btau,\vv)\in\VV\times\QQ$.
In view of the coercivity estimate~\eqref{coercivityVQ} established in Lemma~\ref{pseudocoercitivity}, the well-posedness of problem~\eqref{eq:2.7b} follows from the Lax--Milgram theorem.  Because problems \eqref{eq:3.1} and \eqref{eq:2.7b} are equivalent, we conclude that \eqref{eq:3.1} admits a unique solution $(\bsigma,\uu)\in\HH\times\QQ$ satisfying $\bsigma\nn=-\gN$ on $\GN$.
\end{proof}

\section{A {skew-symmetric} Nitsche scheme}
\label{Nitsche}
%
%
%

%With the aim of defining the discrete scheme, we assume for simplicity that the Neumann datum satisfies $\gN\in L^2(\GN)^n$. Consequently, the duality pairing between $\HH_{00}^{-1/2}(\GN)$ and $\HH_{00}^{1/2}(\GN)$ can be replaced by the $L^2(\GN)^n$ inner product.

Let $\{\mathcal{T}_h\}_{h>0}$ be a shape-regular family of partitions of
$\overline{\Omega}$ into triangles for $n=2$ and tetrahedra for $n=3$. Let
$h_T$ denote the diameter of an element $T\in\mathcal{T}_h$, and define
$\displaystyle h \defi \max_{T\in\mathcal{T}_h} h_T$.
Moreover, let $E$ be an edge for $n=2$ or a face for $n=3$, and $\EE_N$ denote the set of all edges/faces of $\T_h$ lying on  $\GN$.
  For  $T\in\T_h$, we define the local tensor-valued Raviart--Thomas space of order $k\geq 0$ by
\[
\tRT{k}(T)
\defi
\big\{
\btau\in L(T)^{n\times n}:
\btau_i\in\RT{k}(T),\;\forall i\in\{1,\dots ,n\}
\big\}\,,
\]
where $\btau_i$ denotes the $i$-th row of the tensor $\btau$, and
$
\RT{k}(T)
\defi
\PP_k(T)^n\oplus \xx\,\PP_k(T)
$
is the classical Raviart--Thomas space of order $k$ on $T$ (see, e.g.,~\cite[Chapter~3]{D08} and \cite[Section~3.3]{Gatica2014}). Here, $\PP_k(\calO)$ denotes the space of polynomials of total degree at most $k$ on $\calO\subset\RR^n$.

The finite element subspaces of $\HH$ and $\QQ$ employed in this work are defined, respectively, by:
\begin{align}
\label{discretespace1}
  \HH_{h,k}
  \defi
  & \left\{
  \btau_h\in \HH \;:\;
  \btau_h|_T\in\tRT{k}(T),\quad \forall T\in\T_h
  \right\},
  \\
  \label{discretespace2}
  \QQ_{h,k}
  \defi
  & \left\{
  \vv_h \in C(\overline{\OO})^n  \;:\;
  \vv_h|_T \in  \PP_{k+1}(T)^n, \quad  \forall T\in\T_h
  \right\},
\end{align}
with $k\geq 0$,
and we equip the product space $\HH_{h,k} \times \QQ_{h,k} $ with the norm:
\begin{equation}
%    \nonumber 
%    && \nh{(\btau_h,\vv_h)}
%    \\
    \label{norm}
 \nh{(\btau_h,\vv_h)}   \defi \left\{ 
    \dfrac{1}{\nu}\, \|\btau_h^d\|_{0,\OO}^2
    + \left\| \nabla \cdot \btau_h\right\|^2_{0,\OO}
    + \sum_{E\in \EE_N} \gamma_N h_E \, \| \btau_h \nn\|^2_{0,E}
    %% + |\vv_h|_{1,\OO}^2
    %% + \|\vv_h\|_{0,\OO}^2
    + \|\vv_h\|_{1,\OO}^2
    + \| \vv_h\|_{0,\GD}^2 
    \right\}^{1/2}.
%    \label{norma}
\end{equation}

To formulate the discrete scheme, we assume, for simplicity, that the Neumann datum satisfies $\gN\in L^2(\GN)^n$. Under this assumption, the duality pairing involving $\gN\in\HH_{00}^{-1/2}(\GN)$ and functions in $\HH_{00}^{1/2}(\GN)$ coincides with the $L^2(\GN)^n$ inner product.
We now introduce the following discrete scheme, incorporating {skew-symmetric} Nitsche's terms to impose the essential boundary conditions of problem \eqref{eq:2.1}: \emph{Find $(\bsigma_h,\uu_h) \in \HH_{h,k} \times \QQ_{h,k}$  such that }
\begin{equation}
 \label{NitscheScheme} 
 A_N((\bsigma_h,\uu_h),(\btau_h,\vv_h))=  F_N(\btau_h,\vv_h)\,,
 \end{equation}
 for all $(\btau_h,\vv_h) \in \HH_{h,k} \times \QQ_{h,k} $, where   $A_N:(\HH_{h,k} \times \QQ_{h,k} )\times(\HH_{h,k} \times \QQ_{h,k}) \longrightarrow \RR$ is the bilinear form defined by
\begin{align}
\nonumber
 A_N\left( (\bsigma_h,\uu_h),(\btau_h,\vv_h) \right) \defi &  A\left( (\bsigma_h,\uu_h),(\btau_h,\vv_h) \right)
 % \dfrac{1}{\nu}\,\displaystyle \int_{\OO} \bsigma_h^d:\btau_h^d +  \int_{\OO} \uu_h \cdot (\nabla \cdot \btau_h) -
 %\displaystyle   \int_{\OO} \vv_h \cdot (\nabla \cdot \bsigma_h) +\\ \nonumber
 %&\, \kappa_1 \,  \displaystyle\int_{\OO} \left(\nabla \uu_h- \dfrac{\bsigma_h^d}{\nu}\right):\left(\nabla\vv_h -\dfrac{\btau_h^d}{\nu}\right)+
  %\kappa_2 \,   \int_{\OO} (\nabla \cdot \bsigma_h) \cdot (\nabla \cdot \btau_h)    +\kappa_3 \, \int_{\GD} \uu_h\cdot \vv_h ,\\
  - \displaystyle \underbrace{\int_{\GN} \btau_h \nn\cdot \uu_h}_{\text{{Consistency} %Consistence 
term}}  
  +\\ 
 &\,   \underbrace{\int_{\GN} \bsigma_h \nn\cdot \vv_h}_{\text{{Skew-symmetric} term}}  +
 \underbrace{ \sum_{E\in \EE_N} \gamma_N h_E \int_{E} ( \bsigma_h \nn)\cdot  (\btau_h \nn)}_{\text{Penalty term}}\,,
  \label{formAN}
\end{align}
for all $(\bsigma_h,\uu_h),(\btau_h,\vv_h) \in \HH_{h,k} \times \QQ_{h,k} $, and $F: \HH_{h,k} \times \QQ_{h,k}  \longrightarrow \RR$ is the linear form defined by
\begin{align}
\nonumber
  F_N\left(\btau_h,\vv_h\right) \defi & \, 
    F\left(\btau_h,\vv_h\right) 
%  \displaystyle\int_{\GD} \btau_h \nn\cdot \ggg +\displaystyle\int_{\OO}  \ff \cdot \vv_h- \kappa_2 \,   \int_{\OO} \ff \cdot (\nabla \cdot \btau_h) 
%      +\kappa_3 \, \int_{\GD} \ggg\cdot \vv_h 
     -\int_{\GN} \gN\cdot \vv_h -
  \sum_{E\in \EE_N} \gamma_N h_E \int_{E} \gN \cdot  (\btau_h \nn)  \,,
\end{align}
for all $(\btau_h,\vv_h) \in \HH_{h,k} \times \QQ_{h,k}  $, with $\kappa_1,$ $\kappa_2$ and $\kappa_3$ positive parameters that satisfy the conditions \eqref{condicionkappa} 
%\franz{(I removed the theta here)}
and $\gamma_N>0$ being {the} Nitsche parameter.

The following result guarantees the existence and uniqueness of a solution to the discrete problem~\eqref{NitscheScheme}.

\begin{theorem} 
\label{stability}
Under the assumptions of Lemma \ref{pseudocoercitivity}, problem~\eqref{NitscheScheme} admits a unique solution $(\bsigma_h,\uu_h)\in\HH_{h,k}\times\QQ_{h,k}$.
\end{theorem}
%
% Proof
%
\begin{proof}
Using the definition of the bilinear form $A_N(\cdot,\cdot)$ and the norm $\nh{(\cdot,\cdot)}$ (cf. \eqref{formAN} and \eqref{norm}, respectively), together with Lemma~\ref{pseudocoercitivity}, there exists a positive constant $C_A \defi \min\left\{ \dfrac{\kappa_2}{2\theta},\theta \right\}$ such that
\begin{equation}
\label{constant-C}
 A_N((\btau_h,\vv_h),(\btau_h,\vv_h))
 \ge C_A\,\nh{(\btau_h,\vv_h)}^2\,,
\end{equation}
for all $(\btau_h,\vv_h)\in\HH_{h,k}\times\QQ_{h,k}$, which proves the result.
\end{proof}
\begin{remark}
 In the case of homogeneous Dirichlet data, $\gD=\zero$, we obtain the stability estimate
 \[ \nh{(\bsigma_h,\uu_h)} \le C\, \left\{ \|\ff\|_{0,\OO}^2 + \|\gN\|_{0,\GN}^2+
 \sum_{E\in \EE_N} \gamma_N h_E\, \| \gN\|_{0,E}^2  \right\}^{1/2},\]
 where the constant $C$ depends on $\theta$ and $\|\balpha\|_{\infty,\OO}$, but is independent of the viscosity $\nu$.
For nonhomogeneous Dirichlet data, the present theoretical framework does not provide a stability estimate for the Nitsche scheme. Establishing such an estimate requires additional arguments and may therefore constitute an interesting direction for future work.
\end{remark}
%
%
%
%
%\begin{remark}
%Since the existence and uniqueness result relies only on the coercivity of the bilinear form
%$A_N(\cdot,\cdot)$, it can be readily extended to other finite element subspaces of
%$\HH\times\QQ$. For instance, one may employ Brezzi--Douglas--Marini
%($\mathbb{BDM}_k$) elements of order $k$ to approximate the tensor
%$\bsigma_h$, together with continuous Lagrange elements
%$\PP_k^n$ for the velocity $\uu_h$,
%or other computationally cheaper finite element subspaces of $\HH\times\PP_k^n$. 
%However, the convergence analysis requires the use of the enriched space $\QQ_{h,k}$, whose velocity component is based on $\PP_{k+1}^n$.
%\end{remark}
%\textcolor{red}{Llegue hasta ac\'a: Viernes 31 Julio 2026 (17:27 hrs)}

%
%
%
\section{Error analysis} \label{error}
We begin by establishing the consistency of the discrete formulation \eqref{NitscheScheme}.
\begin{lemma}
  \label{lemma-orto}
For $s>1/2$,  assume $(\bsigma,\uu)\in H^{s}(\OO)^{n \times n} \cap \HH  \times   H^{1+s}(\OO)^{n}\cap \QQ$
  solution of \eqref{eq:2.1}, and  $(\bsigma_h,\uu_h) \in \HH_{h,k} \times \QQ_{h,k}$ solution of  the discrete scheme \eqref{NitscheScheme}. Then 
  \[
    A_N((\bsigma-\bsigma_h,\uu-\uu_h),(\btau_h,\vv_h))=0\,,
  \]
  for all $(\btau_h,\vv_h)\in\HH_{h,k}\times\QQ_{h,k}$. 
  % \franz{We need some extra regularity to define the penalty term for the continuous weak solution (the same as for the a priori bound should be fine).}
\end{lemma}
%
% Proof
%
\begin{proof}
The claim follows directly from \eqref{NitscheScheme}, integration by parts,
and the extra regularity assumptions of $(\bsigma,\uu)\in H^{s}(\OO)^{n \times n} \cap \HH  \times   H^{1+s}(\OO)^{n}\cap \QQ$, which satisfies \eqref{eq:2.1}.
\end{proof}
%
%
%
%
%
% We define the space \[\ZZ \defi \left\{\btau \in L^2(\OO)^{n\times n}: \,  \right\}.\]

Now, we introduce some interpolation operators, together with their corresponding approximation properties, which will be instrumental in establishing the optimal convergence of the proposed method. For each $T\in \TT_h$, we define $\PTk:L^2(T)^n \longrightarrow \PP_k(T)^n$ as the $L^2$-orthogonal projector. Then, the following approximation estimates hold  (see \cite[Theorem~18.16]{EG21-FE1}): % (see \cite[Proposition~1.134]{EG04}):
  \begin{align}
    \label{proyector1}
    \|\vv - \PTk(\vv)\|_{0,T} \le&\, C\, h_T |\vv|_{1,T}\,,
    \\[4px]
    \label{proyector2}
    |\vv - \PTk(\vv)|_{1,T} \le&\, C\, |\vv|_{1,T}\,,
  \end{align}
  for all $\vv\in H^1(T)^n$.  To estimate the pseudostress error, we consider, for $s\in (1/2,k+1]$, the global and local Raviart--Thomas interpolation operators $\Pi_h^k: H^{s}(\OO)^{n \times n} \cap\HH \longrightarrow \HH_{h,k}$ and $\Pi_T^k(\btau) \defi \Pi_h^k(\btau)|_T$, respectively, characterized by the following identities (for details, see \cite{BookBrezziFortin2013,Gatica2014}):
  \begin{equation}
  \label{CHRT1}
\int_E \Pi_T^{k}(\btau)\nn \cdot \boldsymbol{q} = \int_E \btau\nn \cdot \boldsymbol{q}\qquad \forall E\in \EE(T), \, \forall T\in\TT_h, \forall \boldsymbol{q}\in \PP_k(E)^n,\, k\ge 0\,,
\end{equation}
  and
    \begin{equation}
  \label{CHRT2}
\int_T \Pi_h^{k}(\btau) : \bsigma = \int_T \btau  : \bsigma   \qquad \forall T\in\TT_h, \forall \bsigma\in \PP_{k-1}(T)^{n \times n},\, k\ge 1\,.
\end{equation}
%From the characterization \eqref{CHRT1} of the Raviart--Thomas interpolant, we have the follow conditions:
From \eqref{CHRT1}, it follows that
\begin{equation}
  \label{aux:prop2}
  \int_E (\btau - \Pi_T^k (\btau))\nn\cdot \PTk(\vv)=0
  \qquad
  \forall\btau\in H^{k+1}(T)^{n\times n}\,,
  \forall\vv\in L^2(E)^n\,,
\end{equation}
for all edge/face $E$ of the element $T\in\TT_h$. % (see \cite{BookBrezziFortin2013,Gatica2014}).
Moreover, for $s\in (0,k+1]$ the interpolation operator satisfies the following estimate (see \cite{GaticaGaticaMarquez2014} for details):
\begin{align}
  \label{RT1b}
  \| \btau - \Pi_h^k (\btau)  \|_{\ddiv,\OO} \le&\,  C\, h^{s}\,
  \Bigg\{ \| \btau\|_{s,\OO}+\| \nabla \cdot \btau\|_{s,\OO} \Bigg\}\,,
\end{align}
for all $\btau \in H^{s}(\OO)^{n \times n} \cap\HH $ with $\nabla \cdot \btau \in H^{s}(\OO)^{n}$.\\

The following result establishes a new interpolation error estimate for $\Pi_h^k$ on the boundary $\GN$, which will be used to control the Nitsche term in the convergence analysis of the scheme.
\begin{lemma}
 \label{lemmatecnico}
 Let $s\in(1/2,k+1]$ %$\medio <s \le k+1$
 and $\bsigma \in H^{s}(\OO)^{n \times n} \cap\HH $. Then, there is a positive constant $C$, such that
 \begin{align}
  \label{RT2b}
  \left\{\sum_{E\in \EE_N} h_E\, \left\| \left(\btau -\Pi_h^k (\btau)\right)\nn\right\|_{0,E}^2 \right\}^{1/2}\le C h^s\, \|\btau\|_{s,\OO}\,.
  \end{align}
\end{lemma}
%
% Proof
%
\begin{proof}
  We begin with the case $s\in(1/2,k+1]$. %$\medio < s < 1$.
  For $E\in \EE_N$, let $\PEk:L^2(E)^n \longrightarrow \PP_k(E)^n$ be the $L^2$-orthogonal projector. Then, by the definition of $\PEk$ and the characterization \eqref{CHRT1} of $\Pi_h^k$ on each $E\in \EE_N$, we have
\begin{equation}
\label{A1}
( \Pi_h^k \btau)\nn = \PEk(\btau \nn) \qquad \forall \btau \in H^{s}(\OO)^{n \times n} \cap\HH.
\end{equation}
Now, let $T\in \T_h$ be the unique element such that $E\in \EE(T) \cap \EE_N$. We define $Q_0 \defi \dfrac{1}{|T|}\displaystyle \int_T \btau$. It is clear that $Q_0\nn \in \PP_0(E)^n \subseteq \PP_k(E)^n$. Hence, by the best-approximation property of the orthogonal projector $\PEk$ and \eqref{A1}, we have
\begin{equation}
 \| (\btau-\Pi_h^k \btau)\nn\|_{0,E} = \| (I-\PEk)( \btau\nn)\|_{0,E} 
=
  \inf_{\vv \in \PP_k(E)^n} \| \btau\nn-\vv\|_{0,E} 
  \le \| \btau\nn-Q_0\nn\|_{0,E} 
    \label{A2}
  \le \| \btau-Q_0\|_{0,E}.
\end{equation}
%
%
%
%\begin{align}
%\nonumber
% \| (\btau-\Pi_h^k \btau)\nn\|_{0,E} =&\, \| (I-\PEk)( \btau\nn)\|_{0,E} \\
% \nonumber=&\,
%  \inf_{\vv \in \PP_k(E)^n} \| \btau\nn-\vv\|_{0,E} \\
%   \nonumber
%  \le&\, \| \btau\nn-Q_0\nn\|_{0,E} \\
%    \label{A2}
%  \le&\, \| \btau-Q_0\|_{0,E}\,.
%\end{align}
%
%
%
To proceed with the proof, we recall the following fractional trace inequality \cite[Lemma 2.1]{AlvArmDur24}: For all $\vv\in  H^{s}(\OO)^n$,
%$s\in (1/2,1)$, 
we have
\[ \|\vv\|_{0,E} \le C\,\Bigg\{ 
h_T^{-1/2}\,\|\vv\|_{0,T}+h_T^{s-1/2}\,|\vv|_{s,T} 
\Bigg\}\,,\]
and the fractional Poincar\'e inequality (for details, see \cite[Lemma~3.26]{EG21-FE1} or \cite[Lemma~7.1]{EG2017}): For all $\btau \in  H^{s}(\OO)^{n\times n}$, we get
\[ \|\btau - Q_0\|_{0,T}\le C h_T^s\, |\btau|_{s,T}\,.\]
Combining these estimates with \eqref{A2}, we obtain
\begin{equation*}
\nonumber
 \| (\btau-\Pi_h^k \btau)\nn\|_{0,E}   \le\, C\, \Bigg\{ 
h_T^{-1/2}\, \| \btau-Q_0\|_{0,T} +h_T^{s-1/2}\, | \btau-Q_0|_{s,T}\Bigg\} 
  \le\,
C h_T^{s-1/2}\, | \btau|_{s,T}\,.
\end{equation*}
The case $s=1$ is a direct consequence of estimate \cite[(25)]{GaticaGaticaMarquez2014}. For the remaining case, when $s\in(1/2,k+1]$, let $r \in \NN$ be such that $1\le r\le k$ and $r<s\le r+1$. By \cite[Lemma~3.7]{HolstLicht2023}, there exists a positive constant $C$ and a matrix $Q_r \in \PP_r(T)^{n \times n}$, such that
\begin{equation}
\label{asterisco1}
|\btau-Q_r|_{1,T} \le C h_T^{s-1}|\btau|_{s,T} \qquad \forall \btau \in H^s(T)^{n\times n}\,.
\end{equation}
If  define $\overline{Q}_r \defi Q_r + \dfrac{1}{|T|} \displaystyle \int_T(\btau-Q_r)$, we have $\displaystyle  \int_T \overline{Q}_r =\int_T \btau$, and therefore $
\overline{Q}_r -\btau \in L_0^2(T)^{n\times n}$.
Next, as $h_T < C$, by Poincar\'e inequality on $L_0^2(T)^{n\times n}$  (cf. \cite[Lemma 3.27]{EG21-FE1}) and \eqref{asterisco1}, we have
\begin{equation}
 \|\btau - \overline{Q}_r\|_{0,T} \le  C\, h_T    |\btau-\overline{Q}_r|_{1,T} 
 %\nonumber
 = C\, h_T    |\btau-Q_r|_{1,T} 
 \le  C\, h_T^{s-1}|\btau-Q_r|_{s,T}\,.
 \label{algun47}
\end{equation}
In turn, since $r \le k$ and $\overline{Q}_r \in \PP_r(T)^{n\times n} \subset \mathbb{RT}_k(T)$, we have $\Pi_h^k \overline{Q}_r=\overline{Q}_r$.
Finally, applying the result for $s=1$ together with \eqref{asterisco1} and \eqref{algun47}, we obtain
%\begin{align}
%\nonumber
%  \| (\btau-\Pi_h^k \btau)\nn\|_{0,E} \le&\,
%  \nonumber
%   \| (\btau - \overline{Q}_r)\nn - \Pi_h^k(\btau-\overline{Q}_r)\nn\|_{0,E} \\
%   \nonumber
%   \le &\,C\, h_E^{1/2}    \| \btau - \overline{Q}_r\|_{1,T}\\
%    \le&\,  C\, h_T^{s-1/2}|\btau-Q_r|_{s,T}\,,
%\end{align}
\[  \| (\btau-\Pi_h^k \btau)\nn\|_{0,E} \le
   \| (\btau - \overline{Q}_r)\nn - \Pi_h^k(\btau-\overline{Q}_r)\nn\|_{0,E} 
   \le \,C\, h_E^{1/2}    \| \btau - \overline{Q}_r\|_{1,T}\\
    \le\,  C\, h_T^{s-1/2}|\btau-Q_r|_{s,T}\,, \]
which completes the proof.
\end{proof}

%  
%   To estimate the pseudostress error, we consider the global and local Raviart--Thomas interpolants 
%$\Pi_h^k: H^{l+1}(\OO)^{n \times n} \cap\HH \longrightarrow \HH_{h,k}$  and $\Pi_T^k (\btau) \defi \Pi_h^k (\btau)|_T$, respectively, which satisfy: For all $T\in \TT_h$
%\begin{align}
%  \label{RT1}
%  |\btau - \Pi_T^k (\btau)|_{m,T} \le&\,  C\, h_T^{l-m+1}\,|\btau|_{l+1,T}\,,
%  \\
%  \label{RT2}
%  \|\nabla \cdot \btau - \nabla \cdot \Pi_T^k (\btau)  \|_{0,T} \le&\,  C\, h_T^{l+1}\,|\nabla \cdot \btau|_{l+1,T}\,,
%\end{align}
%for all $\btau \in H^{l+1}(T)^{n \times n}$, with  $\nabla \cdot \btau \in H^{l+1}(T)^n$, $0\le m \le l+1$, $0\le l \le k$.
 Finally, for the velocity approximations, we consider the Lagrange interpolant $\II_h^k: H^{s}(\OO)^n \longrightarrow \QQ_{h,k}$, which satisfies: For all $T \in \TT_h$
\begin{equation}
  % |\etu|_{l,T} \defi
  |\vv-\II_h^k (\vv) |_{l,T} \le  C h_T^{s-l}\, |\vv|_{s,T}\,,
\label{LAGRANGE}
\end{equation}
for all $\vv \in H^{s}(T)^n$ with $l=0,1$ and  $ \dfrac{n}{2}< s\le k+2$  (for details, see 
\cite{BookBrezziFortin2013,BS08,EG21-FE1}).
%\cite{BookBrezziFortin2013,BS08,EG21-FE1,Gatica2014}).

Theorem~\ref{main} provides an optimal \emph{a priori} error estimate in the energy norm \eqref{norm}, with a constant $C$ that is independent of $h$.
With the aim of making the dependence on the physical parameters at least partially explicit, the following result is established under assumptions on the parameters of problem $(P)$ that are typically satisfied in practice, namely:
\begin{equation}
 \label{parameter-model-dependences}
 \nu \le 1 \quad \text{and} \quad 1 \le  \|\balpha\|_{\infty,\OO} \le \theta\,.
\end{equation}
\begin{theorem}
\label{main}
Consider assumptions (H1)-(H2), the parameter conditions
\eqref{condicionkappa} on $\kappa_1$, $\kappa_2$, and $\kappa_3$, with
$\gamma_N>0$ and $k\ge0$, together with the conditions
\eqref{parameter-model-dependences}. Moreover, assume $s\in(1/2,k+1]$ and that the solution
of \eqref{eq:2.1} satisfies  $(\bsigma,\uu)\in H^{s}(\OO)^{n \times n} \cap \HH  \times   H^{s+1}(\OO)^{n}\cap \QQ$
with $\nabla\cdot\bsigma\in H^{s}(\Omega)^n$, and let
$(\bsigma_h,\uu_h)\in\HH_{h,k}\times\QQ_{h,k}$ be the solution of
\eqref{NitscheScheme}. Then, there exists a positive constant $C$,
independent of $h$, but depending on $\theta$, $\kappa_1$, $\kappa_2$, and $\gamma_N$, such that
\begin{equation}
  \nonumber
  \nh{(\bsigma-\bsigma_h,\uu-\uu_h)} \le  
   C\,h^{s}\,      
           \dfrac{2\theta}{\kappa_2} \max\left\{
      \dfrac{1}{\nu^{1/2}},\theta\right\}\, 
              \Bigg\{ 
             \dfrac{1}{\nu}\,\left\| \bsigma\right\|_{s,\OO}
    +   \dfrac{1}{\nu}\,\| \nabla \cdot \bsigma\|_{s,\OO}+\theta\,\| \uu\|_{s+1,\OO}
          \Bigg\}\,.
\end{equation}

%  \begin{equation}
%  \label{estima_error}
%    \nh{(\bsigma-\bsigma_h,\uu-\uu_h)} \le  
%  C\, h^{s}
%  \Bigg[ \left( 1 + \dfrac{1}{\nu^{1/2}}+ \dfrac{1}{\nu}\right)\,\left\| \bsigma\right\|_{s,\OO}  +
% \| \nabla \cdot \bsigma\|_{s,\OO}+\| \uu\|_{s+1,\OO}  \Bigg].
% \end{equation}
%
%
%
%
\end{theorem}

%\franz{(Maybe a comment about the dependency in $\nu$, particularly for $\nu$ small ?}
%\textcolor{red}{Llegu\'e hasta ac\'a (Lunes 12 de enero 18:57 hrs)}\\
%
%
%
\begin{proof} We proceed as in the proof of \cite[Theorem~5]{CCHP26}. To this end, we define  $\etsgma \defi \bsigma - \Pi_h^k (\bsigma)$ and 
$\etu    \defi \uu-\II_h^k (\uu).$ 
Recalling the definition of the bilinear form $A_N(\cdot,\cdot)$ (cf.~\eqref{formAN}) and applying~\eqref{aux:prop2} together with the Cauchy--Schwarz inequality, we obtain
  \begin{align}
    \nonumber
    &
    A_N((\etsgma,\etu),(\btau_h,\vv_h))
    \\ \nonumber
    % 
    % 1
    % 
    =&
    \dfrac{1}{\nu}\,\displaystyle \int_{\OO} (\etsgma)^d:\btau_h^d
    + \int_{\OO} \etu \cdot (\nabla \cdot \btau_h)
    - \int_{\OO} \vv_h \cdot (\nabla \cdot \etsgma)
    + \theta\, \int_{\OO} \etu\cdot \vv_h
    +\int_{\OO}(\nabla\etu)\balpha\cdot\vv_h
    \\ \nonumber
    &\,
    + \kappa_1 \,  \displaystyle\int_{\OO} \left(\nabla \etu- \dfrac{(\etsgma)^d}{\nu}\right):\left(\nabla\vv_h +\dfrac{\btau_h^d}{\nu}\right)
    + \dfrac{\kappa_2}{2\theta}\,\int_{\OO} (\nabla \cdot \etsgma-\theta\,\etu-(\nabla\etu)\balpha) \cdot (\nabla\cdot\btau_h+\theta\,\vv_h+(\nabla\vv_h)\balpha)
    \\ \nonumber
    &\,
    + \kappa_3 \, \int_{\GD} \etu \cdot \vv_h
    - \int_{\GN} (\btau_h \nn)\cdot \etu
    + \sum_{E\in \EE_N}  \int_{E} (\etsgma \nn)\cdot \left(\vv_h - \PTk(\vv_h)\right)
    + \sum_{E\in \EE_N} \gamma_N h_E \int_{E} ( \etsgma \nn)\cdot  (\btau_h \nn)\\ \nonumber
        % 
    % 2
    % 
    \le&
    \, C\,
    \Bigg\{
    \dfrac{1}{\nu}\,\| (\etsgma)^d\|_{0,\OO}^2
    + \|\etu\|^2_{0,\OO}
    + \|\nabla \cdot \etsgma\|_{0,\OO}^2
    + \theta^2\,\|\etu\|^2_{0,\OO}
    + \|(\nabla\etu)\balpha\|_{0,\OO}^2
    + \left\|\nabla \etu- \dfrac{(\etsgma)^d}{\nu} \right\|_{0,\OO}^2
    \\ \nonumber
    &
    \quad
    + \dfrac{1}{\theta^2}\left\|\nabla \cdot \etsgma-\theta\,\etu-(\nabla\etu)\balpha\right\|_{0,\OO}^2
    + \| \etu\|_{0,\GD}^2
    +\,\sum_{E\in \EE_N} h_E^{-1}  \|\etu \|_{0,E}^2
    + \sum_{E\in \EE_N} h_E   \|\etsgma \nn\|_{0,E}^2 
    \Bigg\}^{1/2}
    \\ \nonumber
    &
    \Bigg\{
    \dfrac{1}{\nu}\,\| \btau_h^d\|_{0,\OO}^2
    + \| \nabla \cdot \btau_h \|^2_{0,\OO}
    + \| \vv_h \|_{0,\OO}^2
    +\, \left\| \nabla\vv_h +\dfrac{\btau_h^d}{\nu} \right\|_{0,\OO}^2
    + \|\nabla\cdot\btau_h+\theta\,\vv_h+(\nabla\vv_h)\balpha\|_{0,\OO}^2
    \\ \nonumber
    &
    \quad
    + \| \vv_h\|_{0,\GD}^2
    + \sum_{E\in \EE_N} \gamma_N h_E \|\btau_h \nn\|_{0,E}^2
    + \sum_{E\in \EE_N} h_E^{-1}\left\|\vv_h - \PTk(\vv_h) \right\|_{0,E}^2 
    \Bigg\}^{1/2}.
    \end{align}
Now, using~\eqref{trilinearprop-desigualdad0}, Young's inequality, and a trace inequality, we obtain
  \begin{align}
    \nonumber
    &
    A_N((\etsgma,\etu),(\btau_h,\vv_h))
    \\ \nonumber
    \le&
    \, C\, \Bigg\{
    \dfrac{1}{\nu}\,\left( 1 + \dfrac{1}{\nu}\right)\,\| (\etsgma)^d\|_{0,\OO}^2
    + \left(1+\dfrac{1}{\theta^2}\right)\|\nabla\cdot\etsgma\|_{0,\OO}^2
    + \left(1+\theta^2+\|\balpha\|_{\infty,\Omega}^2+\dfrac{\|\balpha\|_{\infty,\Omega}^2}{\theta^2}\right)\|\etu\|^2_{1,\OO}
    \\ \nonumber
    &\,
    + \sum_{E\in \EE_N} h_E^{-1} \|\etu \|_{0,E}^2
    + \sum_{E\in \EE_N} h_E \|\etsgma \nn\|_{0,E}^2
    \Bigg\}^{1/2}
    \\ \nonumber
    &
    \,\Bigg\{
      \dfrac{1}{\nu}\left(1+\dfrac{1}{\nu}\right)\,\|\btau_h^d\|_{0,\OO}^2
    + \| \nabla \cdot \btau_h \|^2_{0,\OO}
    + (1+\theta^2+\|\balpha\|_{\infty,\OO}^2)\| \vv_h \|_{1,\OO}^2
    + \,\| \vv_h\|_{0,\GD}^2
    \\ \nonumber
    &
    +\,\sum_{E\in \EE_N} \gamma_N h_E \|\btau_h \nn\|_{0,E}^2
    +\, \sum_{E\in \EE_N} h_E^{-1} \left\|\vv_h - \PTk(\vv_h) \right\|_{0,E}^2
    \Bigg\}^{1/2}\\ \nonumber
     % 
    % 4
    % 
    \le&
    \, C\, \Bigg\{
    \dfrac{1}{\nu^2}\, \| (\etsgma)^d\|_{0,\OO}^2
    + \|\nabla\cdot\etsgma\|_{0,\OO}^2
    + \theta^2\, \|\etu\|^2_{1,\OO}
   % \\ \nonumber
   % &\,
    + \sum_{E\in \EE_N} h_E \|\etsgma \nn\|_{0,E}^2
    + \sum_{E\in \EE_N} h_E^{-1} \|\etu \|_{0,E}^2
    \Bigg\}^{1/2}
    \\ \nonumber
    &
    \,\Bigg\{
      \dfrac{1}{\nu^2}\,\|\btau_h^d\|_{0,\OO}^2
    + \| \nabla \cdot \btau_h \|^2_{0,\OO}
    + \theta^2\, \| \vv_h \|_{1,\OO}^2
    + \,\| \vv_h\|_{0,\GD}^2
    \\
    &
    +\, \sum_{E\in \EE_N} h_E^{-1} \left\|\vv_h - \PTk(\vv_h) \right\|_{0,E}^2
        +\,\sum_{E\in \EE_N} \gamma_N h_E \|\btau_h \nn\|_{0,E}^2
    \Bigg\}^{1/2}\,
    ,
    \label{main:eq1}
  \end{align}
  for all $(\btau_h,\vv_h) \in \HH_{h,k}\times \QQ_{h,k}$.

  In turn, combining the local trace inequality (cf. \cite[Lemma 12.15]{EG21-FE1}) 
  \begin{equation}
    \label{local-trace}
    \| \vv \|_{0,E}^2 \le C\, \left( h_T^{-1} \,\|\vv \|_{0,T}^2  +h_T\,|\vv |_{1,T}^2  \right)\,
    \quad\forall \vv \in H^1(T)^d\,,
  \end{equation}
  and Young's inequality with the mesh regularity and \eqref{proyector1}-\eqref{proyector2}, we deduce
  \begin{align}
    %
   % \nonumber
    \sum_{E\in \EE_N} h_E^{-1}   \left\|\vv_h - \PTk(\vv_h) \right\|_{0,E}^2 
    %
    %
    %
  %  &\le\,C\, \sum_{T\in \TT_h} \left[ h_T^{-2} \, \left\|\vv_h - \PTk(\vv_h) \right\|_{0,T}^2 +\left|\vv_h - \PTk(\vv_h) \right|_{1,T}^2 \right]
   % \\
    \label{main:eq2}
    &\le \, C\, | \vv_h|_{1,\OO}^2\,.
  \end{align}
  Thus,  using these last estimates,  we obtain
  \begin{align}
    \nonumber
    &\dfrac{1}{\nu^2}\,\|\btau_h^d\|_{0,\OO}^2
    + \| \nabla \cdot \btau_h \|^2_{0,\OO}
    + \theta^2\,\| \vv_h \|_{1,\OO}^2
    + \,\| \vv_h\|_{0,\GD}^2
    \\ \nonumber
    &
    +\,\sum_{E\in \EE_N} \gamma_N h_E \|\btau_h \nn\|_{0,E}^2
      +\, \sum_{E\in \EE_N} h_E^{-1} \left\|\vv_h - \PTk(\vv_h) \right\|_{0,E}^2
    \\
    \label{aux:nh}
    &\le\;\max\left\{
     \dfrac{1}{\nu},
      \theta^2\right\}\, 
      \nh{(\btau_h,\vv_h)}^2\,.
  \end{align}
  On the other hand, from the definition of the deviator  $(\etsgma)^d$, we get %of a tensor %and \eqref{RT1}, it follows that
  \begin{equation}
    \label{RT3}
    \| (\etsgma)^d\|^2_{0,T}
    =  \left\| \etsgma - \dfrac{\tr{(\etsgma)}}{n} \II\right\|^2_{0,T}
    \leq 2\,\left\| \etsgma\right\|^2_{0,T} + \dfrac{2}{n}\,\left\| \tr{(\etsgma)}\right\|^2_{0,T}
    \leq C\,\,\|\etsgma\|_{0,T}^2\,.
    %\le C\, h_T^{s}\,|\bsigma|_{s,T}\,.
  \end{equation}
  %and from this and \eqref{RT1}, we get 
  %\begin{align}
  %  \label{RT3}
  %  \| (\etsgma)^d\|_{0,T} \le&\,  C\, h_T^{l+1}\,|\bsigma|_{l+1,T}\,.
  %\end{align}
  %
  %
  %
  Then, applying Lemma \ref{lemmatecnico}, \eqref{RT1b},  \eqref{LAGRANGE}, 
\eqref{local-trace}, \eqref{RT3}, and employing again the mesh regularity, we deduce that
  \begin{align}
    \nonumber
    &
      \, \dfrac{1}{\nu^2}\,\| (\etsgma)^d\|_{0,\OO}^2
      + \|\nabla\cdot\etsgma\|_{0,\OO}^2
      + \theta^2\, \|\etu\|_{1,\OO}^2+ \sum_{E\in \EE_N} h_E   \|\etsgma \nn\|_{0,E}^2
      + \sum_{E\in \EE_N} h_E^{-1}  \|\etu \|_{0,E}^2
    \\ %\nonumber
    % 
    % 1
    % 
%    \le
%    &
%      \, C\, \Bigg\{
%      \dfrac{1}{\nu^2}\,\left\| \etsgma \right\|_{0,\OO}^2
%      + \| \nabla \cdot \etsgma\|_{0,\OO}^2
%      + \theta^2\, \|\etu\|_{1,\OO}^2
%    \\ \nonumber
%      &\,
%      + \sum_{T\in \TT_h} \left[   \|\etsgma \|_{0,T}^2 +h_T^2\, |\etsgma |_{1,T}^2 \right]
%      %+ \| \etu\|^2_{1,\OO} 
%      +\sum_{T\in \TT_h} \left[ h_T^{-2} \|\etu \|_{0,T}^2  +|\etu |_{1,T}^2  \right]
%      \Bigg\}
%    \\
    % 
    \le &
          \, C\,h^{2s}\, 
          \Bigg\{ 
             \dfrac{1}{\nu^2}\,\left\| \bsigma\right\|_{s,\OO}^2
    +\dfrac{1}{\nu^2}\,\| \nabla \cdot \bsigma\|^2_{s,\OO}+\theta^2\,\| \uu\|^2_{s+1,\OO}
          \Bigg\}\,.
    \label{main:eq3}
  \end{align}
  Connecting \eqref{main:eq1}, \eqref{aux:nh} and \eqref{main:eq3}, we get %can deduce that 
  \begin{equation}
    \label{main:eq4}
    A_N((\etsgma,\etu),(\btau_h,\vv_h))
    \le
    C\,h^{s}\,\max\left\{
     \dfrac{1}{\nu^{1/2}},
      \theta\right\}\, 
         \Bigg\{ 
             \dfrac{1}{\nu}\,\left\| \bsigma\right\|_{s,\OO}
    +\dfrac{1}{\nu}\, \| \nabla \cdot \bsigma\|_{s,\OO}+\theta\,\| \uu\|_{s+1,\OO}
          \Bigg\}
    \,\nh{(\btau_h,\vv_h)}\,,
  \end{equation}
  for all $(\btau_h,\vv_h) \in \HH_{h,k}\times \QQ_{h,k}$.
  %
  %
  %
%  \\
%%
%%
%%
%  Next, we define
%  \begin{equation*}
%    \label{def:xi}
%    \xisigma\defi \Pi_h^k (\bsigma)-\bsigma_h
%    \quad\text{ and }\quad
%    \xiu    \defi \II_h^k (\uu)-\uu_h\,,
%  \end{equation*}
%  % 
%  and note that
%  \begin{equation}
%    \label{aux:trig}
%    \nh{(\bsigma-\bsigma_h,\uu-\uu_h)}
%    \le
%    \nh{(\etsgma,\etu)}
%    +
%    \nh{(\xisigma,\xiu)}\,.
%  \end{equation}

 On the other hand,  from \eqref{RT1b}, \eqref{LAGRANGE}, the mesh regularity, and the local and global trace inequalities, we have
  \begin{align}
   % \nonumber
    \nh{(\etsgma,\etu)}^2
    % 
    % 1
    % 
%    =&\, 
%         \dfrac{1}{\nu}\,
%         \|\left(\etsgma\right)^d\|_{0,\OO}^2
%         + \left\| \nabla \cdot \etsgma\right\|^2_{0,\OO}
%         + \sum_{E\in \EE_N} \gamma_N h_E \| \etsgma \nn\|^2_{0,E}
%         + \|\etu\|_{1,\OO}^2
%         + \|\etu\|_{0,\GD}^2
%    \\
    \label{main:eq5}
    % 
    % 2
    % 
    \le &\,C\,h^{2s}\,
          \left[
          %\left(1+
          \dfrac{1}{\nu}
          %\right)
          \,
           \|\bsigma\|^2_{s,\OO}
          +\|\nabla\cdot\bsigma\|^2_{s,\OO}
          +\|\uu\|_{s+1,\OO}^2
          \right]\,.
  \end{align}

%   \begin{align}
%     \nonumber
%     \nh{(\xisigma_h,\xiu_h)}^2
%     &\le C_A^{-1}\,A_N((\xisigma_h,\xiu_h),(\xisigma_h,\xiu_h))
%     \\
%     \nonumber
%     &= -C_A^{-1}\,A_N((\etsgma,\etu),(\xisigma_h,\xiu_h))
%     \\
%     \nonumber
%     &\le C\,h^{s}\,\calM^{1/2}\nh{(\xisigma_h,\xiu_h)}\,,
%   \end{align}
Now, defining
$\left( \xisigma_h,\xiu_h \right)\defi$
$\left(\Pi_h^k(\bsigma)-\bsigma_h,\,
\II_h^k(\uu)-\uu_h\right),$ and using the coercivity of $A_N(\cdot,\cdot)$ (cf.~\eqref{constant-C}) 
with $C_A^{-1}= \dfrac{2\theta}{\kappa_2}$, 
and
Lemma~\ref{lemma-orto}, 
we obtain
      \begin{equation}
     \nh{(\xisigma_h,\xiu_h)}^2
\le \dfrac{2\theta}{\kappa_2} \,A_N((\xisigma_h,\xiu_h),(\xisigma_h,\xiu_h))
     \\
     \nonumber
= -\dfrac{2\theta}{\kappa_2}\,A_N((\etsgma,\etu),(\xisigma_h,\xiu_h))\,,
   \end{equation}
   and applying \eqref{main:eq4}, we conclude that
   \begin{equation}
     \label{nhxi}
     \nh{(\xisigma,\xiu)}
     \le   C\,h^{s}\,
     \dfrac{2\theta}{\kappa_2} \max\left\{
      \dfrac{1}{\nu^{1/2}},\theta\right\}\, 
              \Bigg\{ 
             \dfrac{1}{\nu}\,\left\| \bsigma\right\|_{s,\OO}
    +\dfrac{1}{\nu}\,\| \nabla \cdot \bsigma\|_{s,\OO}+\theta\,\| \uu\|_{s+1,\OO}
          \Bigg\}\,.
   \end{equation}

   The proof is completed noting that 
  \begin{equation*}
%    \label{aux:trig}
    \nh{(\bsigma-\bsigma_h,\uu-\uu_h)}
    \le
    \nh{(\etsgma,\etu)}
    +
    \nh{(\xisigma,\xiu)}\,,
  \end{equation*}
  and then combining the above estimate with
 \eqref{main:eq5} and \eqref{nhxi}.
% \\
%\textcolor{red}{Llegue hasta ac\'a: Martes 4 de agosto 2026 (17:44 hrs)}\\
%\textcolor{red}{Para la pr\'oxima revisar las cuentas anteriores!!}\\
\end{proof}

\begin{remark}
An important difference between Theorem~\ref{main} and the corresponding result in~\cite{CCHP26} lies in the regularity assumptions imposed on the exact solution $(\bsigma,\uu)$. More precisely, thanks to Lemma~\ref{lemmatecnico}, the stronger assumption
$   (\bsigma,\uu)\in H^{k+1}(\OO)^{n \times n} \cap \HH \times H^{k+2}(\OO)^{n}\cap \QQ $
required in~\cite{CCHP26}, can be substantially relaxed, leading to weaker regularity requirements that are essentially minimal for the present analysis.
\end{remark}

%\input{numeric3.tex}
%\textcolor{red}{Llegue hasta ac\'a: Marte 13 Octubre 2026 (21:16 hrs)}

%%%%%%%%%%%%%%%%%%%% 
\section{Numerical experiments}\label{num}
%%%%%%%%%%%%%%%%%%%% 
In this section, we present a series of numerical experiments in two and three dimensions to assess the performance of the proposed method and confirm the convergence rates established in Theorem~\ref{main}. The numerical scheme was implemented using the open-source finite element library \texttt{FEniCS}~\cite{fenics15}, and the resulting linear systems were solved using the direct solver \texttt{MUMPS}~\cite{MUMPS01}. 
The numerical solutions in two dimensions were visualized using \texttt{ParaView}~\cite{Paraview}, whereas the convergence curves were generated using \texttt{Gnuplot}~\cite{gnuplot}. Moreover, in accordance with Lemma \ref{pseudocoercitivity}, we use the following parameters in all numerical examples:
\[
  \kappa_1=0.9\,\nu\,,
  \qquad
  \kappa_2=\frac12\min\left\{ 1 - \dfrac{\eps}{2\theta}, \dfrac{\theta \kappa_1}{n\|\balpha\|^2_{\infty,\OO}}\right\}\,,
  \qquad
  \kappa_3=\|\balpha\cdot\nn\|_{\infty,\GD}\,,
    \qquad\text{ and }\qquad
    \gamma_N=100\,\nu^{-1}.
\]

% To illustrate the numerical convergence of the proposed method, we compute the errors
% \[
%   e_{p}:=\|p-p_h\|_{0,\Omega},\quad
%   e_{\bsigma}:=\|\bsigma-\bsigma_h\|_{\text{div},\Omega},\quad
%   e_{\|\uu\|}:=|\uu-\uu_h|_{0,\Omega},\quad
%   e_{|\uu|}:=|\uu-\uu_h|_{1,\Omega},
% \]
% together with experimental rates of convergence
% \[
%   \text{rate}:=\dfrac{\log(e_{(\cdot)}/e'_{(\cdot)})}{\log(h/h')}
% \]
% where $h$ and $h'$ denote two consecutive mesh sizes with errors $e_{(\cdot)}$ and $e'_{(\cdot)}$, respectively.

%%%%%%%%%%%%%%%%%%%% 
\subsection{Robustness to non-solenoidal convection}
%%%%%%%%%%%%%%%%%%%% 
% This example corresponds to the code oseen_ex8.py
The aim of this example is to assess the performance of the proposed numerical method when the convective field is not solenoidal. To this end, we consider the following manufactured solution of problem (P): %\ref{Porig}:
\begin{equation}
\label{sol_example_1}
  \uu(x,y) \defi
  \begin{pmatrix}
    x(1-x)(1-2y)
    \\
    -(1-2x)y(1-y)
  \end{pmatrix}
  \qquad\text{ and }\qquad
  p(x,y) \defi (x-0.5)(y-0.5),
\end{equation}
on the domain $\OO\defi (0,1)\times(0,1)$.
Dirichlet and Neumann boundary conditions are prescribed on $\GD\defi\{(x,y)\in\partial\OO:x=0\vee y=0\}$ and $\GN\defi\partial\OO\setminus\GD$, respectively. The parameters of the generalized Oseen equations are set to $\theta=1$, $\nu=10^{-1}$, and the convective field is defined by
%\begin{equation*}
%  \balpha(x,y)\defi
%  \uu(x,y)
%  +
%  \dfrac{\varepsilon}{4}
%  \begin{pmatrix}
%    x\\
%    y
%  \end{pmatrix},
%\end{equation*}
$  \balpha(x,y) \defi   \uu(x,y) + \varepsilon(x/4,y/4)$, 
where $\varepsilon>0$ is the parameter appearing in assumption (H1). Since $\nabla\cdot\uu=0$, we have $\nabla\cdot\balpha=\varepsilon/2<\varepsilon$, and hence (H1) is satisfied. Finally, the boundary data $\gD$ and $\gN$, together with the source term $\ff$, are determined from the exact solution given in~\eqref{sol_example_1}.

Figure~\ref{fig:ex1_energy_norm} illustrates the convergence behavior of the proposed method on a sequence of uniformly refined meshes for $\varepsilon\in\{1,10^{-3},10^{-6},0\}$  with  $k=0$ and $k=1$. We observe that the numerical performance remains essentially unaffected as $\varepsilon$ decreases towards zero, corresponding to the transition from non-solenoidal to solenoidal convective fields. Moreover, the observed convergence rates are in agreement with the theoretical predictions of Theorem~\ref{main}. Similarly, Figure~\ref{fig:ex1} shows the convergence behavior of the individual errors in the pseudostress, velocity, and pressure, for which the expected convergence rates are observed across all values of $\varepsilon$, providing further numerical evidence of the robustness of the proposed method with respect to non-solenoidal convection.

\begin{figure}[htbp]
  \centering
  \includegraphics[width=0.45\textwidth]{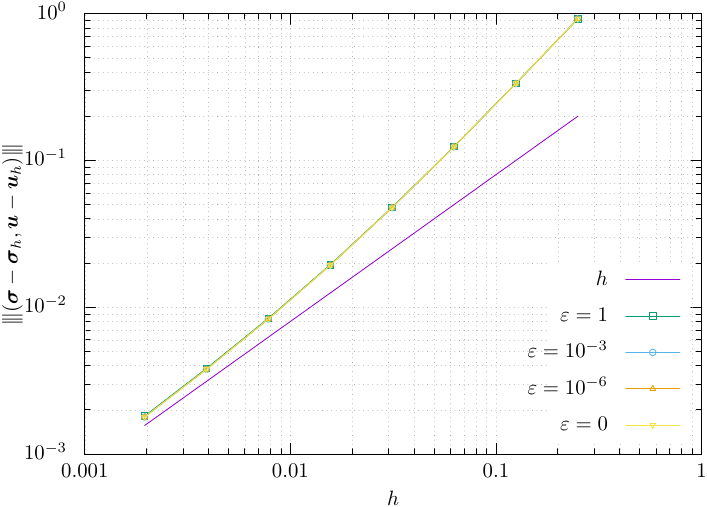}
  \hspace{10px}
  \includegraphics[width=0.45\textwidth]{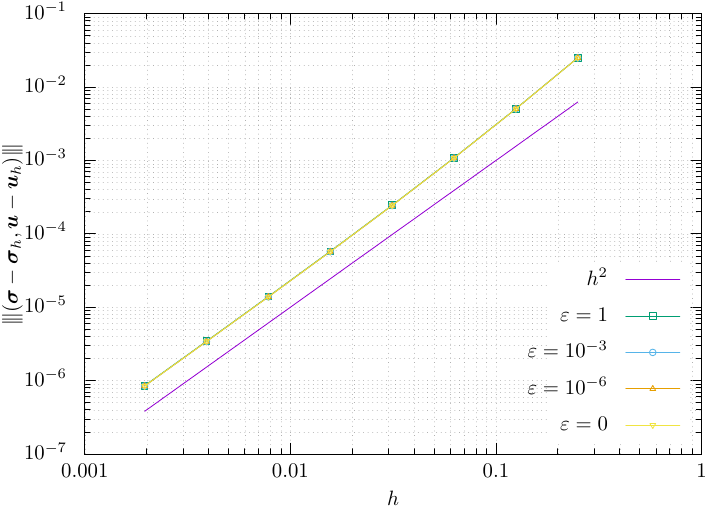}
  \caption{Robustness to non-solenoidal convection. Convergence history of the error energy norm $\nh{(\bsigma-\bsigma_h,\uu-\uu_h)}$,
  for $\varepsilon\in\{1,10^{-3},10^{-6},0\}$, with $k=0$ (left) and $k=1$ (right).}
  \label{fig:ex1_energy_norm}
\end{figure}

\begin{figure}[htbp]
  \centering
  \includegraphics[width=0.45\textwidth]{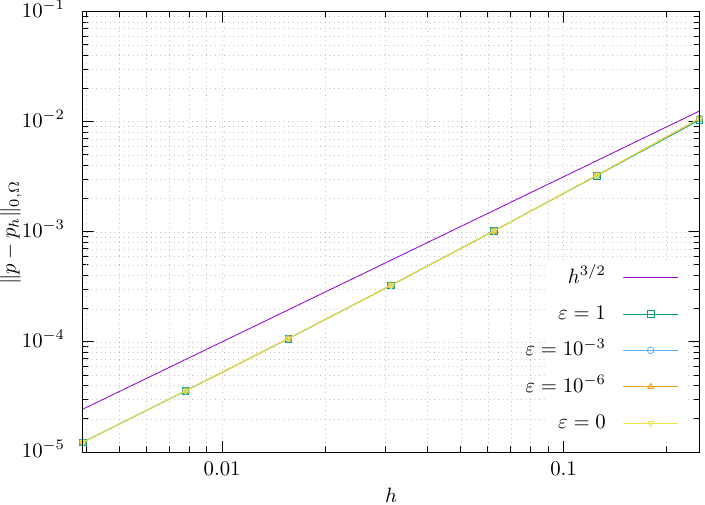}
  \hspace{8px}
  \includegraphics[width=0.45\textwidth]{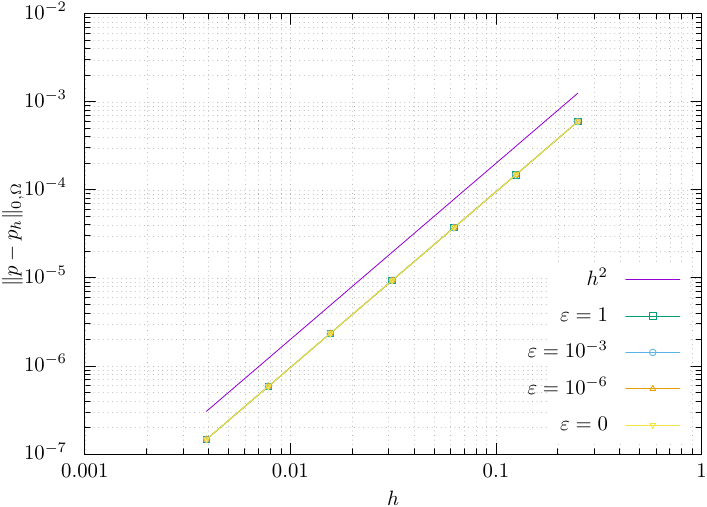}
  \\[8px]
  \includegraphics[width=0.45\textwidth]{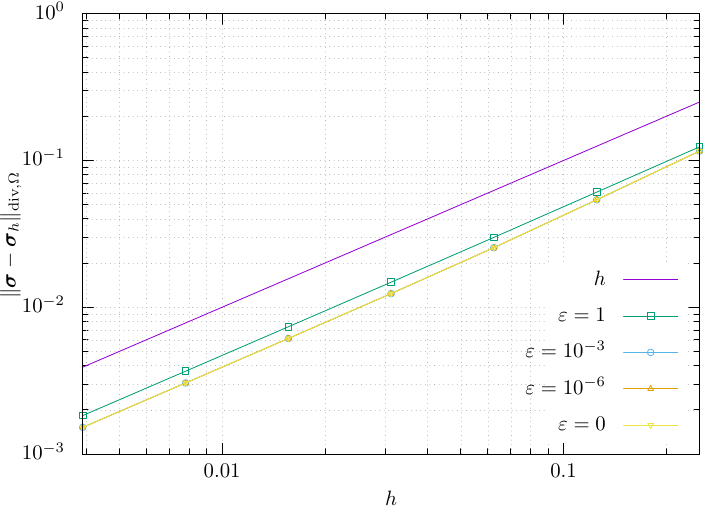}
  \hspace{8px}
  \includegraphics[width=0.45\textwidth]{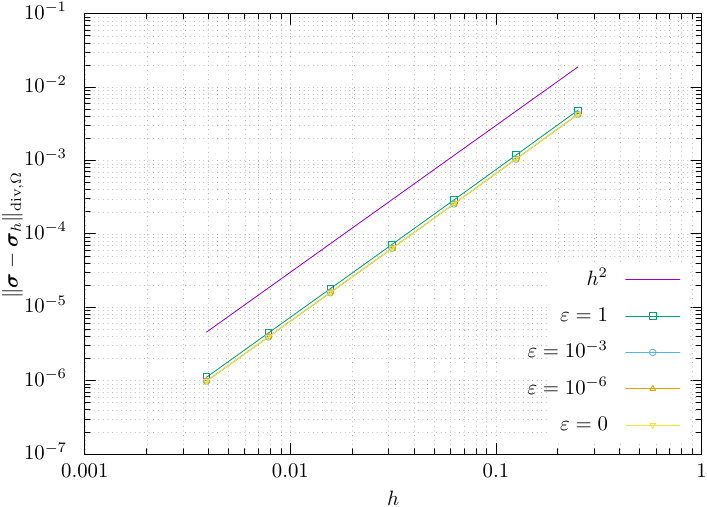}
  %\\[8px]
  %\includegraphics[width=0.45\textwidth]{ex1P1u.pdf}
  %\hspace{8px}
  %\includegraphics[width=0.45\textwidth]{ex1P2u.pdf}
  \\[8px]
  \includegraphics[width=0.45\textwidth]{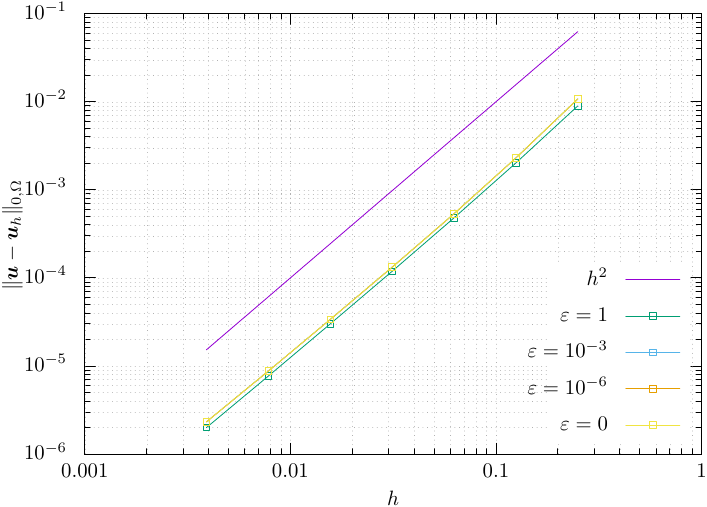}
  \hspace{8px}
  \includegraphics[width=0.45\textwidth]{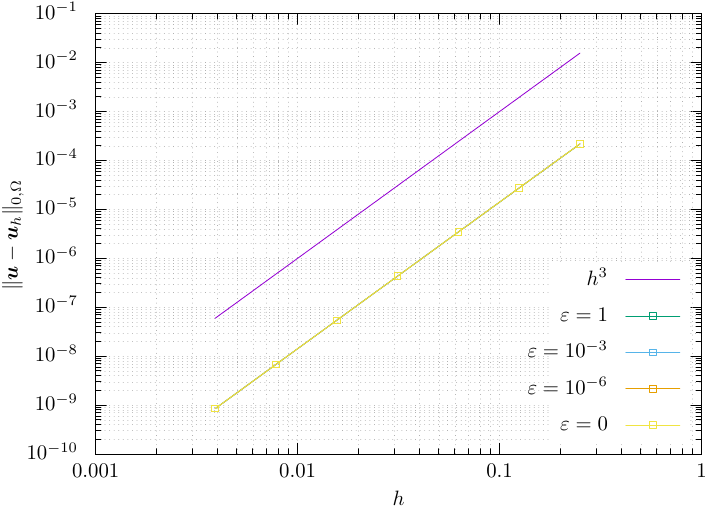}
  \\[8px]
  \includegraphics[width=0.45\textwidth]{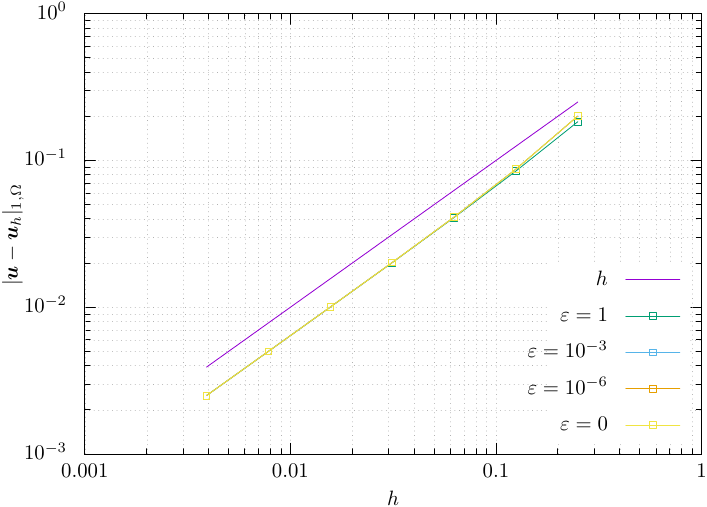}
  \hspace{8px}
  \includegraphics[width=0.45\textwidth]{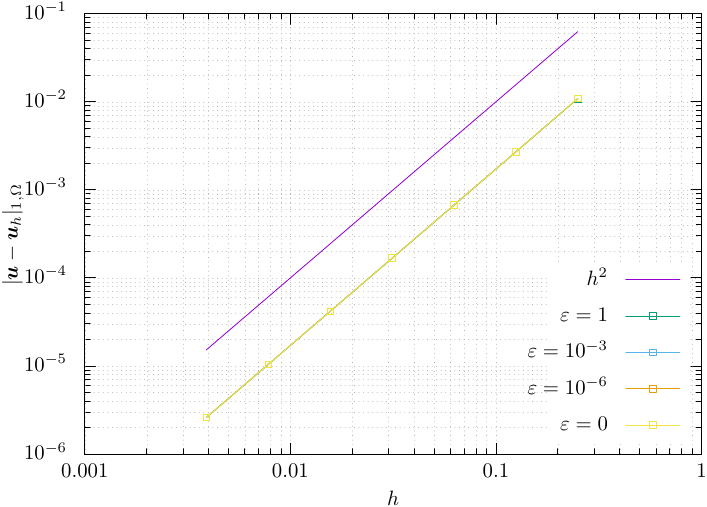}
  \caption{Robustness to non-solenoidal convection. Convergence histories of the errors, displayed from top to bottom, in the $L^2$-norm of the pressure, the $\mathbf{H}(\mathrm{div})$-norm of the pseudostress tensor, and the $L^2$-norm and $H^1$-seminorm of the velocity, for $\varepsilon\in\{1,10^{-3},10^{-6},0\}$, with $k=0$ (left) and $k=1$ (right).}
  \label{fig:ex1}
\end{figure}

% %%%%%%%%%%%%%%%%%%%%
\subsection{An example with boundary layers}
% %%%%%%%%%%%%%%%%%%%%

Inspired by the numerical example presented in \cite{LR26}, we consider the domain $\OO \defi (0,1)\times(0,1)$, with parameters $\balpha=(1,1)$, $\nu=10^{-2}$ and $\theta=1$. The source term $\ff$ is chosen so that the exact solution of problem (P) reads
\begin{equation}
  \uu(x,y)=\left(\dfrac{\partial\phi}{\partial y},-\dfrac{\partial\phi}{\partial x}\right)
  \qquad\text{ and }\qquad
  p(x,y)=e^{x+y}-(e-1)^2\,,
  \label{sol_example_2}
\end{equation}
where
\[
  \phi(x,y)=x^2\left(1-e^{\lambda(x-1)}\right)^2y^2\left(1-e^{\lambda(y-1)}\right)^2,
\]
with $\lambda=\dfrac{1}{2\nu}$.
We impose homogeneous Dirichlet boundary conditions, $\gD=\zero$, on $\GD\defi\{(x,y)\in\partial\OO:x=0\vee y=0\}$, while the Neumann datum $\gN$ on $\GN\defi\partial\OO\setminus\GD$ is determined from the exact solution given in~\eqref{sol_example_2}.

\begin{figure}[htbp]
  \centering
  \includegraphics[width=0.45\textwidth]{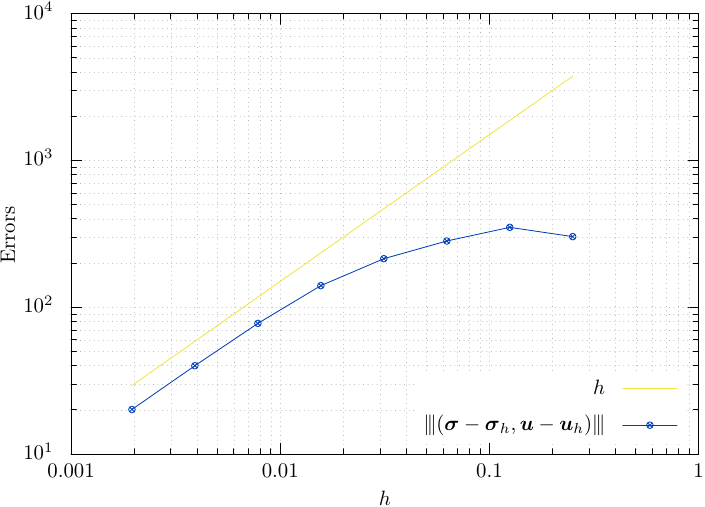}
  \hspace{10px}
  \includegraphics[width=0.45\textwidth]{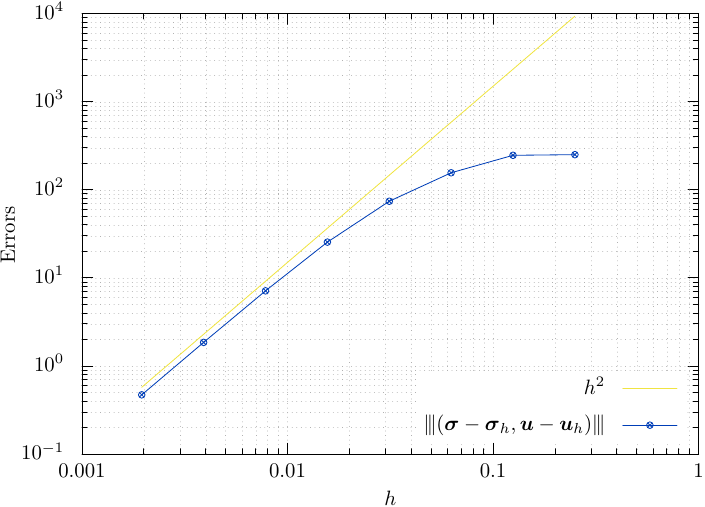}
  \caption{An example with boundary layers. Convergence history of the error energy norm $\nh{(\bsigma-\bsigma_h,\uu-\uu_h)}$ with $k=0$ (left) and $k=1$ (right).}
    \label{fig:ex2_energy_norm}
\end{figure}

\begin{figure}[htbp]
  \centering
  \includegraphics[width=0.45\textwidth]{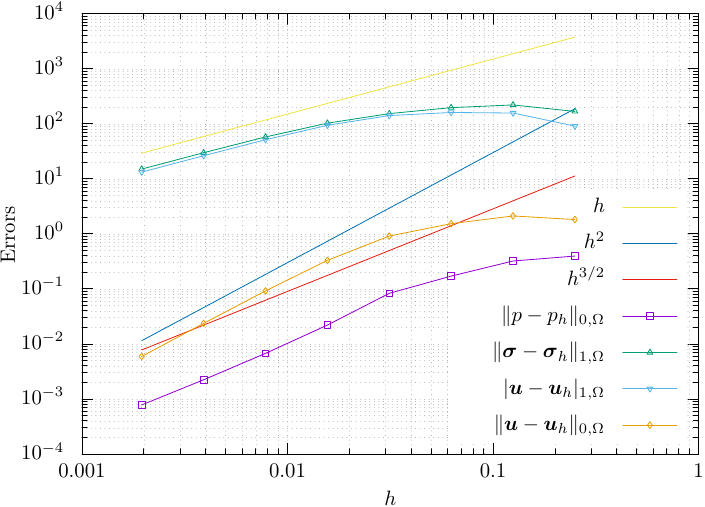}
  \hspace{10px}
  \includegraphics[width=0.45\textwidth]{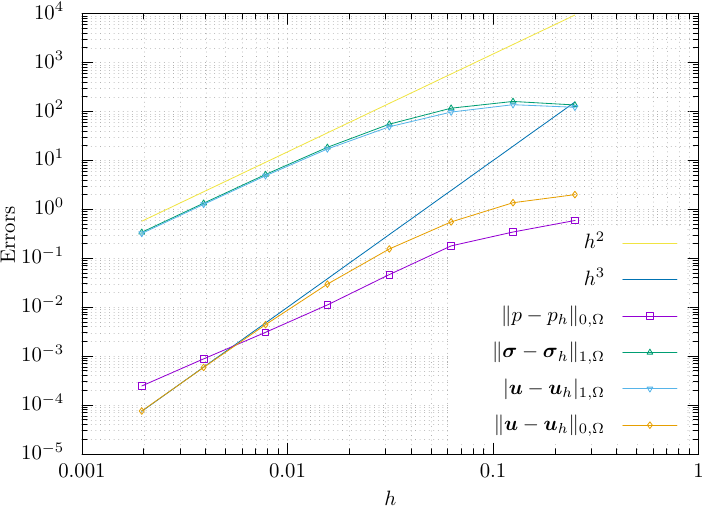}
  \caption{An example with boundary layers. Convergence histories of the errors in the $L^2$-norm of the pressure, the $\mathbf{H}(\mathrm{div})$-norm of the pseudostress tensor, and the $L^2$-norm and $H^1$-seminorm of the velocity, with $k=0$ (left) and $k=1$ (right).}
  \label{fig:ex2}
\end{figure}

The first component of the velocity field $\uu$ exhibits a boundary layer along $y=1$, particularly near the corner $(1,1)$, whereas the second component exhibits a similar behavior along $x=1$. Figure~\ref{fig:ex2_energy_norm} illustrates the convergence behavior of the error in the energy norm $\nh{(\bsigma-\bsigma_h,\uu-\uu_h)}$, for $k=0$ and $k=1$. As expected, the presence of boundary layers affects the convergence behavior on coarse meshes. However, as the mesh is sufficiently refined, the expected convergence rates are recovered, in agreement with the theoretical predictions of Theorem~\ref{main}. A similar behavior is observed in Figure~\ref{fig:ex2} for the individual errors in the pseudostress tensor, velocity, and pressure, with the expected convergence rates recovered for both $k=0$ and $k=1$.
Figure~\ref{fig:ex2-uhph} displays the isovalues of the computed velocity magnitude $|\uu_h|$ and pressure $p_h$, whereas Figure~\ref{fig:ex2-sigmah} shows the four components of the pseudostress tensor $\bsigma_h$. All results were obtained with $k=1$ on a mesh consisting of $1{,}048{,}576$ elements. Only the numerical solutions are presented, since their isovalues are visually indistinguishable from those of the corresponding exact solutions.

\begin{figure}[htbp]
  \centering
  \includegraphics[width=0.35\textwidth]{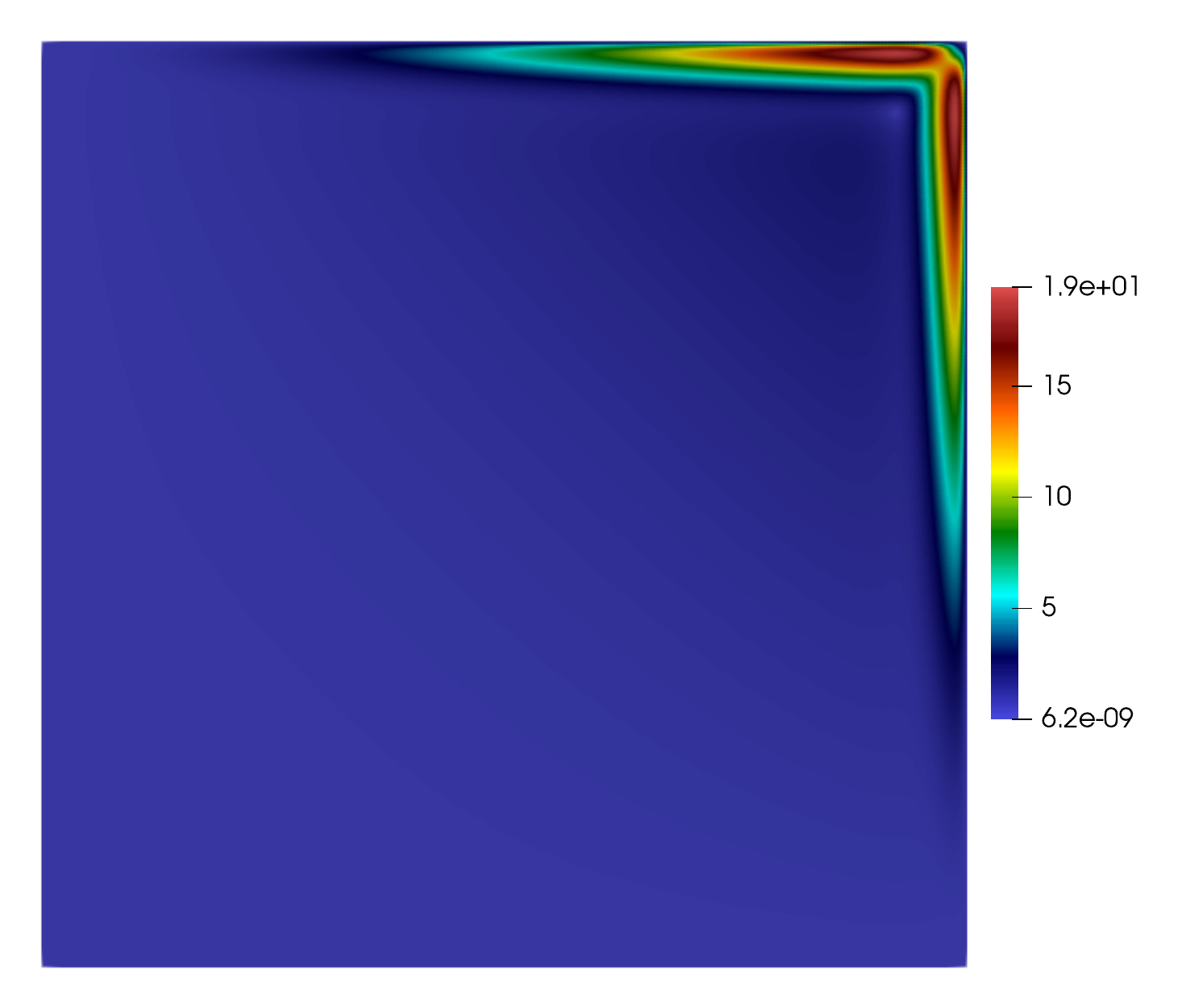}
  \hspace{8px}
  \includegraphics[width=0.35\textwidth]{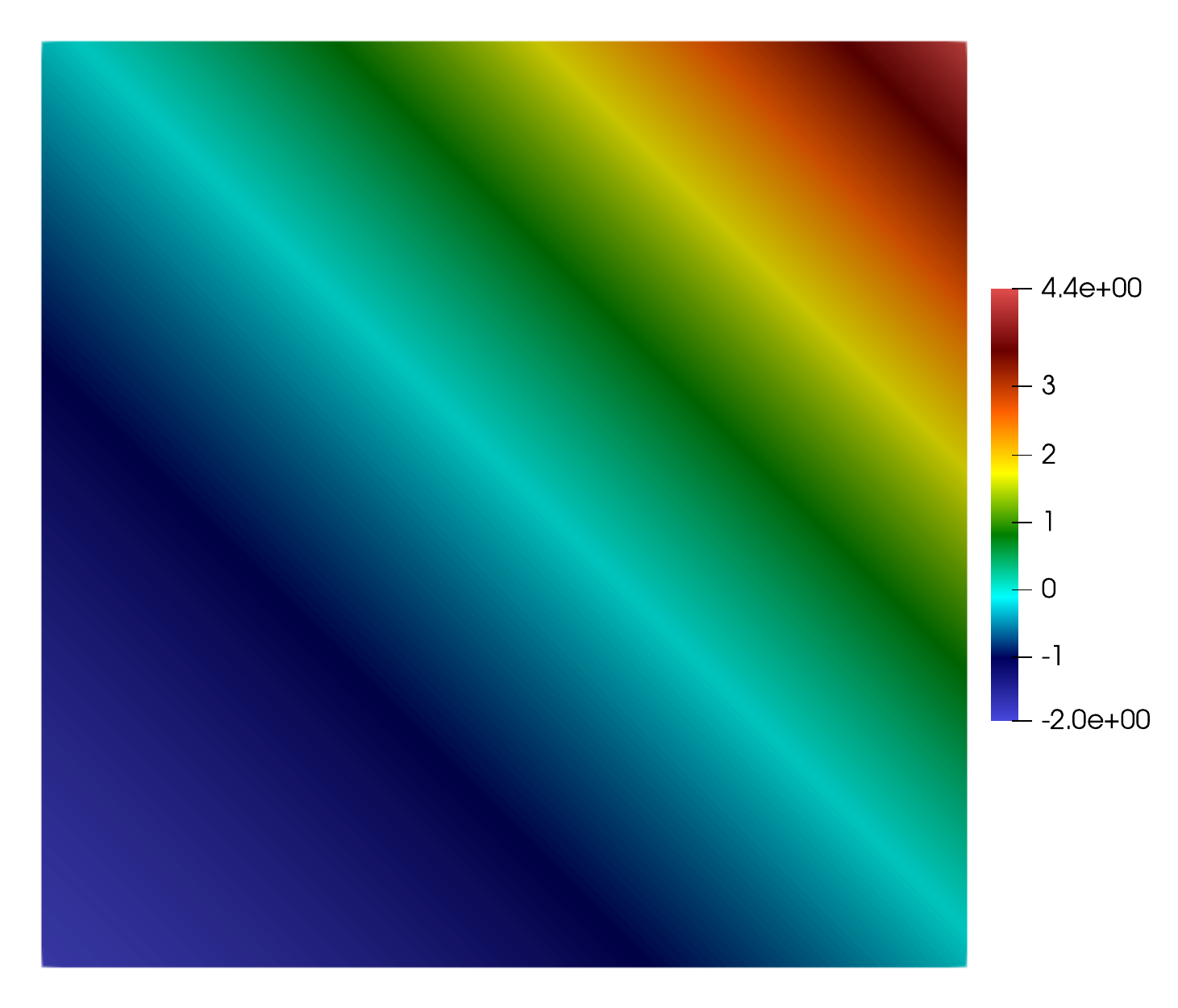}
  \caption{An example with boundary layers. Isovalues of the velocity magnitude $|\uu_h|$ (left) and pressure $p_h$ (right), computed on a mesh consisting of $1{,}048{,}576$ elements, with $k=1$.}
  \label{fig:ex2-uhph}
\end{figure}

\begin{figure}[htbp]
  \centering
  \includegraphics[width=0.35\textwidth]{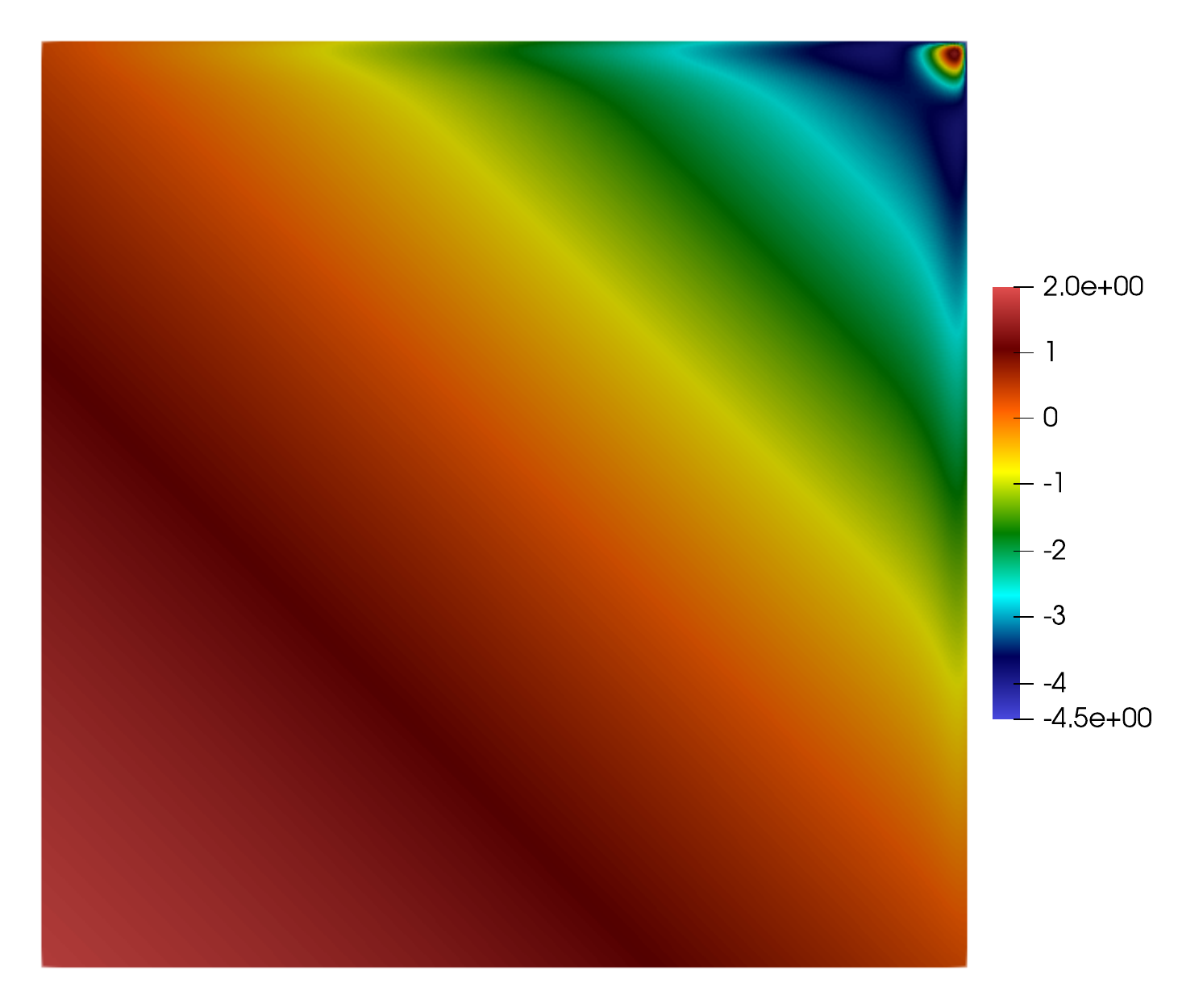}
  \hspace{8px}
  \includegraphics[width=0.35\textwidth]{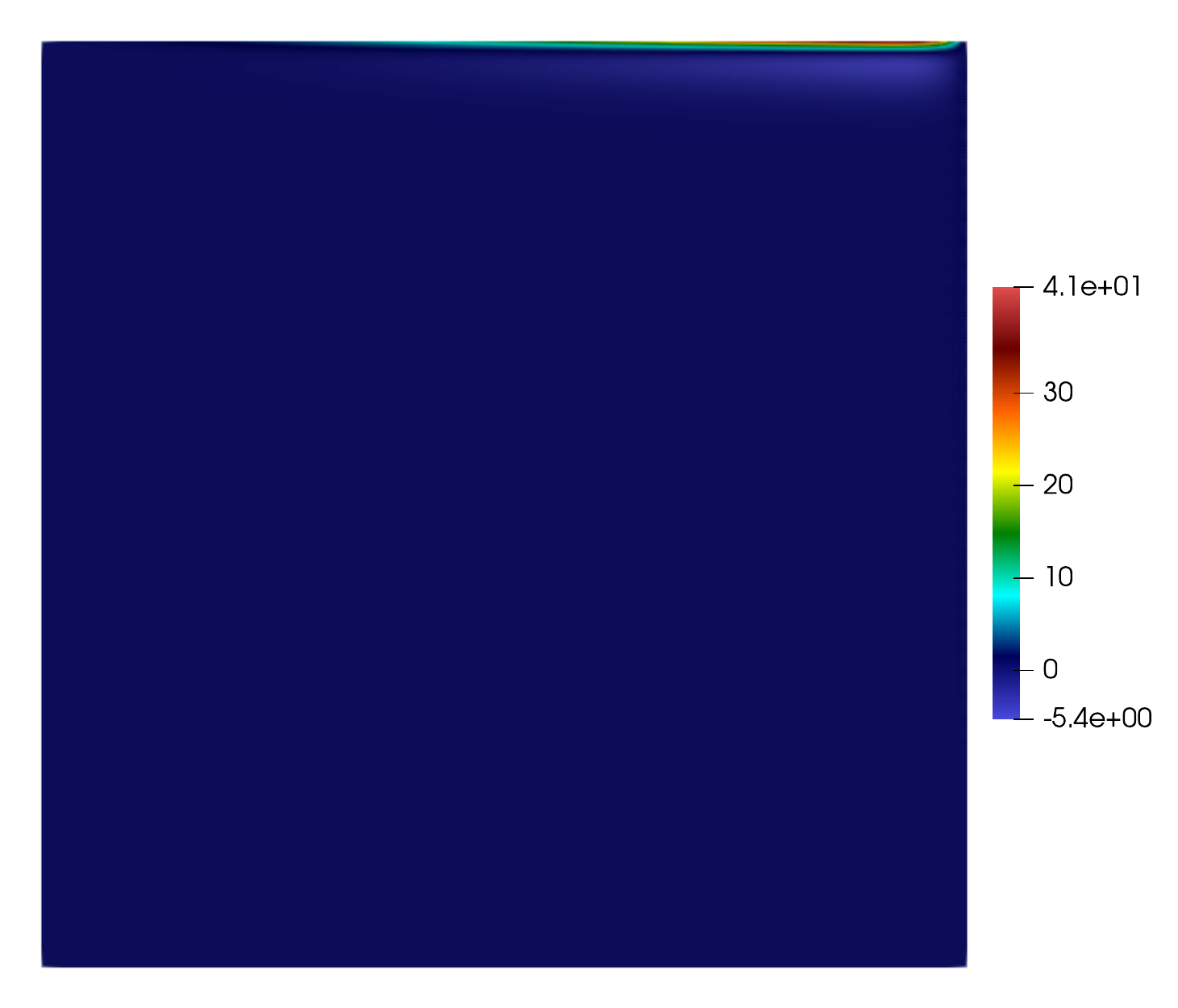}
  \\
  \includegraphics[width=0.35\textwidth]{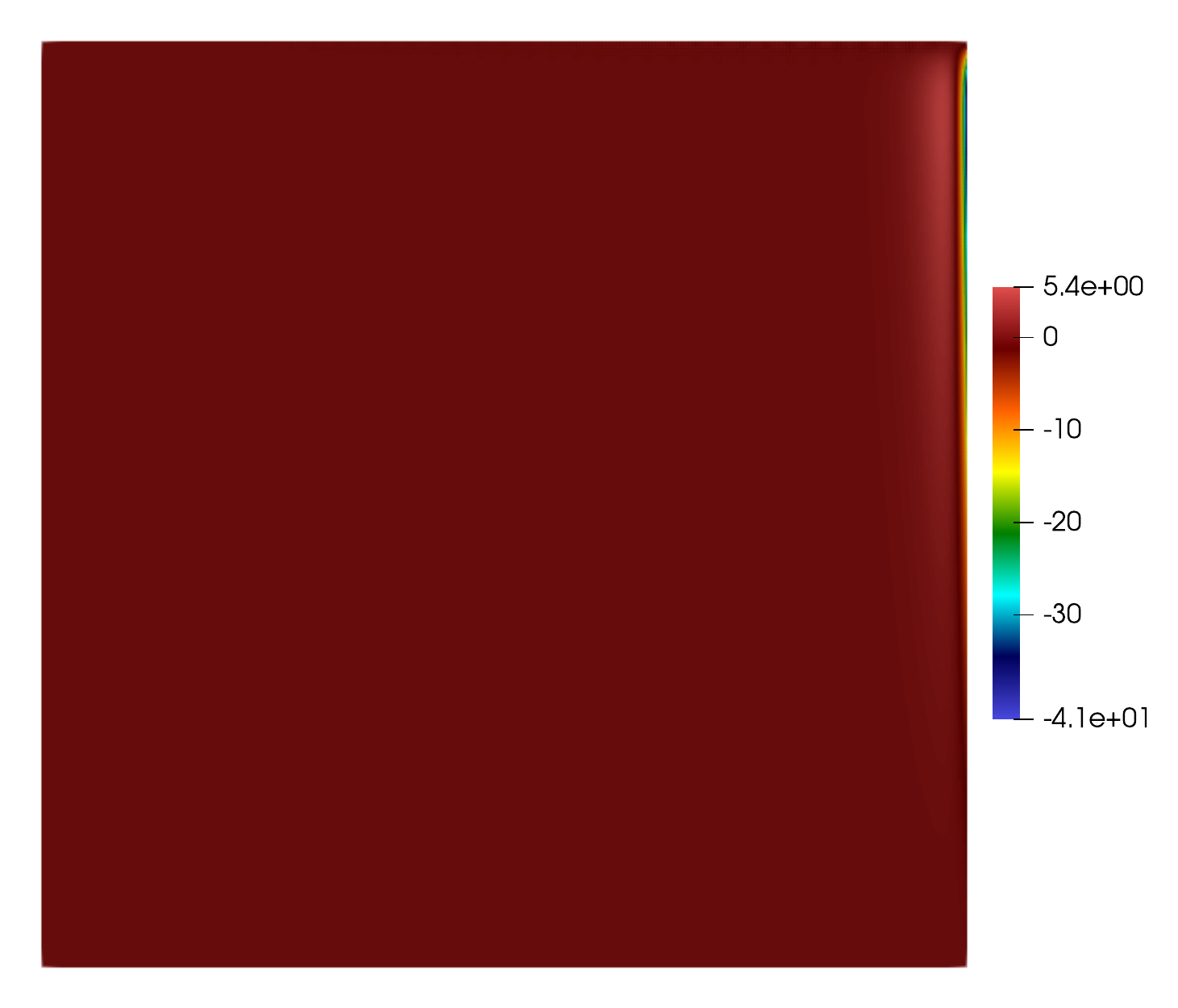}
  \hspace{8px}
  \includegraphics[width=0.35\textwidth]{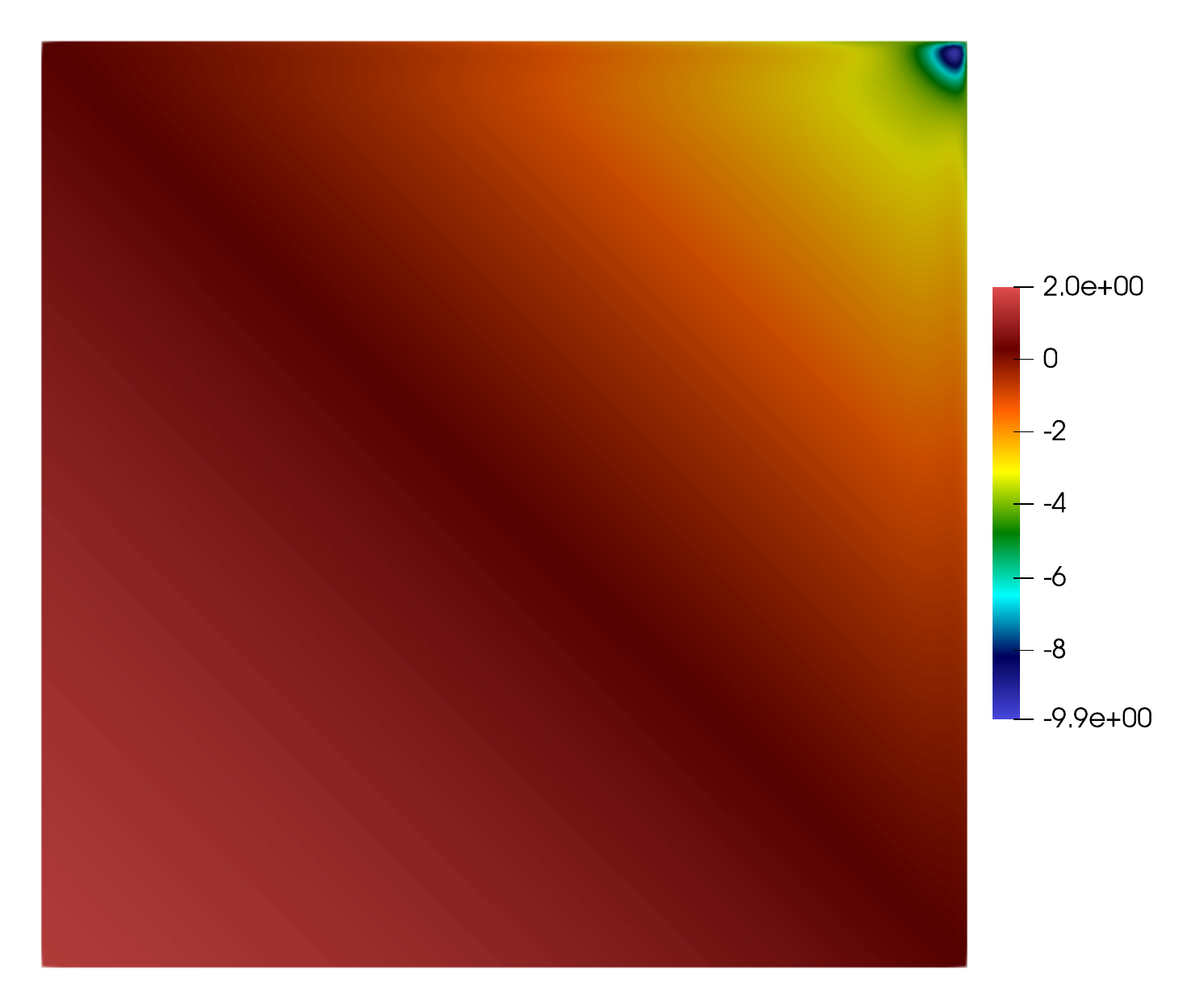}
  \caption{An example with boundary layers. Isovalues of the four components of the computed pseudostress tensor $\bsigma_h$, obtained with $k=1$ on a mesh consisting of $1{,}048{,}576$ elements: $(\bsigma_h)_{11}$ (top left), $(\bsigma_h)_{12}$ (top right), $(\bsigma_h)_{21}$ (bottom left), and $(\bsigma_h)_{22}$ (bottom right).}
  \label{fig:ex2-sigmah}
\end{figure}

% %%%%%%%%%%%%%%%%%%%%
\subsection{A manufactured solution in three dimensions}
The final experiment assesses the performance of the proposed scheme for a three-dimensional problem with a non-solenoidal convective field. To this end, we consider the manufactured solution of problem $(P)$ given by $\uu(x,y,z)=e^{-y}(xz,2z,z^2/2)$ and $p(x,y,z)=(x-0.5)(y-0.5)z^2$, defined on the unit cube $\OO\defi(0,1)^3$. We set $\theta=1.2$, $\nu=10^{-2}$, and define the convective field by $\balpha=\uu+10^{-3}(0,0,z)$. Although $\nabla\cdot\uu=0$, the convective field satisfies $\nabla\cdot\balpha=10^{-3}$, and hence $\|\nabla\cdot\balpha\|_{\infty,\OO}=10^{-3}$. Accordingly, we choose $\varepsilon=1.001\times10^{-3}$ to satisfy assumption (H1). Furthermore, Dirichlet and Neumann boundary conditions are prescribed on $\GD\defi\partial\OO\setminus\GN$ and $\GN\defi\{(x,y,z)\in\partial\OO:y=1\}$, respectively. The boundary data $\gD$ and $\gN$, as well as the source term $\ff$, are determined from the exact solution $(\uu,p)$.
% %%%%%%%%%%%%%%%%%%
%\[
%\mathbf{u}(x, y, z) = e^{-y}\begin{pmatrix} 
%	xz \\[1ex] 
%	2z \\[1ex] 
%	\displaystyle\frac{z^2}{2} 
%\end{pmatrix}
%\quad\text{ and }\quad
%p(x, y, z) = (x - 0.5)(y - 0.5) z^2
%\]8(
%defined on the unit square domain $\OO=(0,1)^3$
%For the numerical tests, we set $\theta=1.2$, $\nu=10^{-2}$, and define the velocity field $\boldsymbol{\alpha}$ as
%\[
%\boldsymbol{\alpha}(x, y, z) = \uu + 10^{-3} \begin{pmatrix} 
%	0 \\ 
%	0 \\ 
%	z 
%\end{pmatrix}.
%\]
%Due to $\Vert \nabla \cdot \boldsymbol{\alpha} \Vert_{\infty,\OO} = 10^{-3}$, in order to satisfy the property (H1), the parameter $\varepsilon$ is assumed as $\varepsilon= 1.001\times 10^{-3}$. 

Figure~\ref{fig:ex3_energy_norm} illustrates the optimal convergence rates in the energy norm established in Theorem~\ref{main}, for $k=0$ and $k=1$, on a sequence of tetrahedral meshes with decreasing mesh size $h$. Similarly, Figure~\ref{fig:ex3} shows the convergence behavior of the individual errors in the pseudostress tensor, velocity, and pressure, for which the expected convergence rates are observed.
\begin{figure}[htbp]
  \centering
  \includegraphics[width=0.45\textwidth]{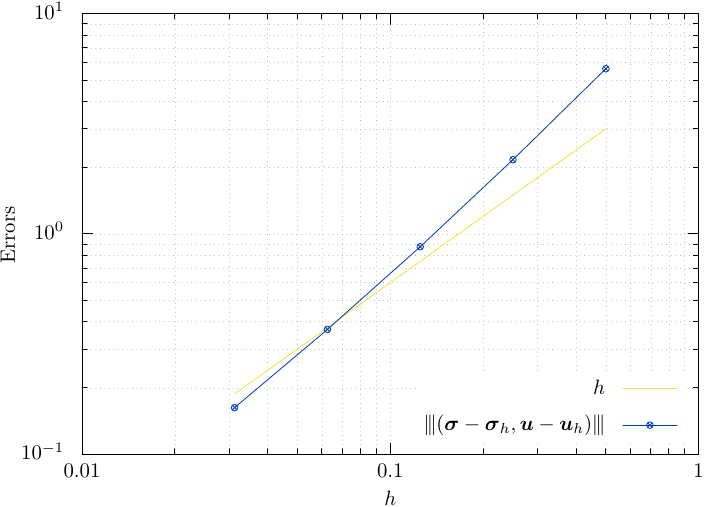}
  \hspace{10px}
  \includegraphics[width=0.45\textwidth]{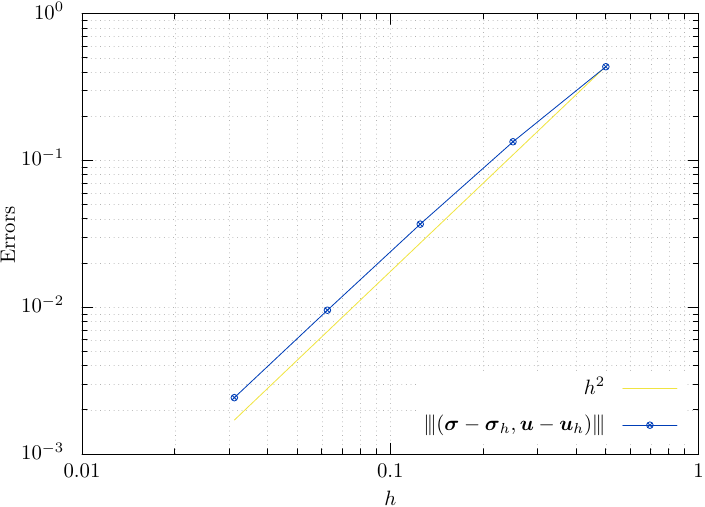}
  \caption{A manufactured solution in three dimensions. Convergence history of the error energy norm $\nh{(\bsigma-\bsigma_h,\uu-\uu_h)}$,
  for $k=0$ (left) and $k=1$ (right).}
  \label{fig:ex3_energy_norm}
\end{figure}
\begin{figure}[htbp]
  \centering
  \includegraphics[width=0.45\textwidth]{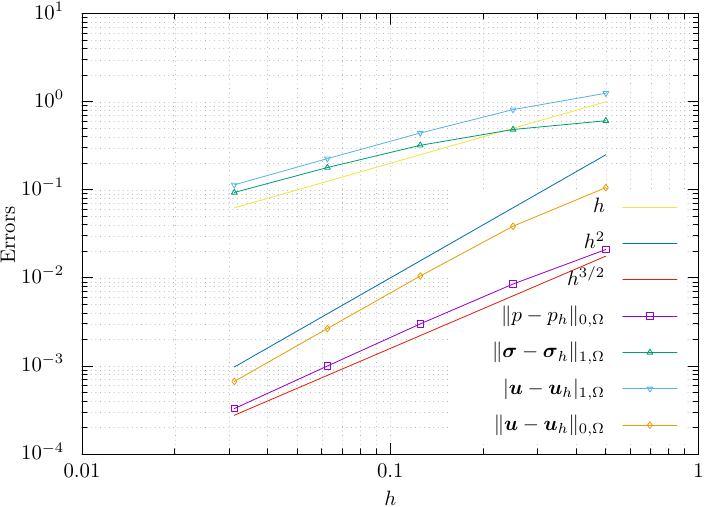}
  \hspace{8px}
  \includegraphics[width=0.45\textwidth]{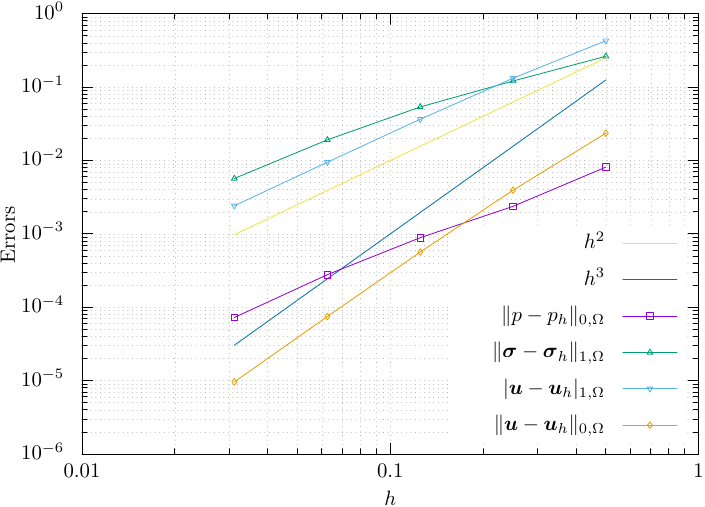}
    \caption{A manufactured solution in three dimensions. Convergence histories of the errors in the $L^2$-norm of the pressure, the $\mathbf{H}(\mathrm{div})$-norm of the pseudostress tensor, and the $L^2$-norm and $H^1$-seminorm of the velocity, with $k=0$ (left) and $k=1$ (right).}
  \label{fig:ex3}
\end{figure}
%
%
%
%%%%%%%%%%%%%%%%%%%%
\section*{Conclusions}
%%%%%%%%%%%%%%%%%%%%
In this work, we have proposed and analyzed a new skew-symmetric Nitsche scheme based on an augmented dual-mixed finite element formulation for the generalized Oseen equations with mixed boundary conditions and non-solenoidal convective fields. %By introducing the pseudostress tensor as an additional unknown, 
%the original problem based in velocity--pressure is transformed  in a mixed formulation pseudostress-velocity by 
%the treatment of boundary conditions differs from that of the classical primal velocity--pressure formulation.
By introducing the pseudostress tensor as a new unknown, the classical velocity-pressure formulation is recast as a dual-mixed pseudostress-velocity formulation, in which the boundary conditions are treated differently from those in the primal formulation.
 In particular, the Dirichlet boundary condition for the velocity is naturally incorporated through integration by parts, whereas the prescribed normal trace of the pseudostress tensor is weakly enforced through consistent Nitsche terms. We have established the well-posedness of the continuous formulation by means of a suitable lifting of the Neumann datum and proved the existence and uniqueness of the discrete solution under appropriate assumptions on the augmentation parameters, without requiring a sufficiently large Nitsche stabilization parameter. Furthermore, the stability of the discrete scheme has been established for homogeneous Dirichlet boundary conditions and general Neumann data. We have also derived optimal a priori error estimates in an energy norm under relatively mild regularity assumptions on the exact solution, thereby extending the analysis to non-solenoidal convective fields. The theoretical results have been supported by numerical experiments in two and three dimensions, implemented using the open-source finite element library \texttt{FEniCS}. The numerical results confirm the theoretical convergence rates and illustrate the performance of the proposed scheme for non-solenoidal convective fields and solutions exhibiting boundary layers. As future research directions, we envisage extending the present framework to the Navier--Stokes equations and developing adaptive finite element procedures based on a posteriori error estimators, with the aim of improving the local accuracy of the numerical approximation while controlling the computational cost.

%%%%%%%%%%%%%%%%%%%%
\section*{Acknowledgements}
%%%%%%%%%%%%%%%%%%%%
The second author was partially 
supported by Direcci\'on de Investigaci\'on of the Universidad Cat\'olica de la Sant\'isima Concepci\'on through project~DIREG 01/2025.
The third author was partially supported by ANID-Chile through grant FONDECYT de Iniciaci\'on No 11241197.

\bibliographystyle{elsarticle-harv}
\bibliography{references}

\end{document}